\documentclass{amsart}
\usepackage[utf8]{inputenc}
\usepackage{bm}
\usepackage{amsfonts}
\usepackage{amsmath}
\usepackage{amssymb}
\usepackage{amstext}
\usepackage{amsthm}
\usepackage{mathtools}
\usepackage{amsopn}
\usepackage{mathrsfs}
\usepackage{tikz}
\usepackage{hyperref}
\usepackage{verbatim}
\usepackage{kbordermatrix}
\usepackage[ruled,vlined]{algorithm2e}

\newtheorem{theorem}{Theorem}
\newtheorem{lemma}{Lemma}[section]
\newtheorem{proposition}[lemma]{Proposition}
\newtheorem{corollary}[lemma]{Corollary}
\theoremstyle{definition}
\newtheorem{definition}[lemma]{Definition}
\theoremstyle{remark}
\newtheorem{remark}{Remark}[section]
\newtheorem{example}[remark]{Example}

\newtheorem{discussion}[remark]{Discussion}
\newtheorem{convention}[remark]{Convention}
\newtheorem{observation}[remark]{Observation}
\newtheorem{question}[remark]{Question}

\numberwithin{equation}{section}

\newcommand{\Z}{\mathbb{Z}}

\newcommand{\K}{\mathbb{K}}

\newcommand{\x}{\mathbf{x}}
\newcommand{\bt}{\mathbf{t}}
\newcommand{\balpha}{\bm{\alpha}}
\newcommand{\bbeta}{\bm{\beta}}
\newcommand{\bgamma}{\bm{\gamma}}
\newcommand{\bT}{\mathbf{T}}
\newcommand{\calI}{\mathcal{I}}
\newcommand{\calB}{\mathcal{B}}
\newcommand{\blgsort}{\text{Borel sort}}
\DeclareMathOperator{\gens}{gens}

\newcommand{\grevlex}{\text{\footnotesize{grevlex}}}
\newcommand{\lex}{\text{\footnotesize{lex}}}
\newcommand{\glex}{\text{\footnotesize{glex}}}

\newcommand{\up}{\text{\small{\textsc{up}}}}
\newcommand{\lw}{\text{\small{\textsc{down}}}}
\newcommand{\lft}{\text{\small{\textsc{left}}}}
\DeclareMathOperator{\bs}{\textsc{BorelSort}}
\DeclareMathOperator{\Borel}{Borel}

\begin{document}

\title{Koszul multi-Rees algebras of principal $L$-Borel Ideals}

\author{Michael DiPasquale}
\address{Department of Mathematics, Colorado State University}
\email{michael.dipasquale@colostate.edu}

\author{Babak Jabbar Nezhad$^1$}
\address{Istanbul, Turkey}
\email{babak.jab@gmail.com}

\thanks{$\phantom{j}^1$ Babak Jabbar Nezhad has also published under the name Babak Jabarnejad~\cite{Jb18}}

\begin{abstract}
	Given a monomial $m$ in a polynomial ring and a subset $L$ of the variables of the polynomial ring, the principal $L$-Borel ideal generated by $m$ is the ideal generated by all monomials which can be obtained from $m$ by successively replacing variables of $m$ by those which are in $L$ and have smaller index.  Given a collection $\mathcal{I}=\{I_1,\ldots,I_r\}$ where $I_i$ is $L_i$-Borel for $i=1,\ldots,r$ (where the subsets $L_1,\ldots,L_r$ may be different for each ideal), we prove in essence that if the bipartite incidence graph among the subsets $L_1,\ldots,L_r$ is chordal bipartite, then the defining equations of the multi-Rees algebra of $\mathcal{I}$ has a Gr\"obner basis of quadrics with squarefree lead terms under lexicographic order.  Thus the multi-Rees algebra of such a collection of ideals is Koszul, Cohen-Macaulay, and normal.  This significantly generalizes a theorem of Ohsugi and Hibi on Koszul bipartite graphs.  As a corollary we obtain that the multi-Rees algebra of a collection of principal Borel ideals is Koszul.  To prove our main result we use a fiber-wise Gr\"obner basis criterion for the kernel of a toric map and we introduce a modification of Sturmfels' sorting algorithm.
\end{abstract}

\maketitle


\section{Introduction}
The Rees algebra of an ideal is a central object of study in commutative algebra.  Geometrically, the Rees algebra of an ideal in a polynomial ring is the coordinate ring of the blowup of projective space along the scheme defined by the ideal. Algebraically, the Rees algebra of an ideal encodes the behavior of all its powers simultaneously.  Similarly, the \textit{multi-Rees} algebra of a family of ideals $I_1,\ldots, I_r$ of a polynomial ring $R$ encodes (geometrically) the coordinate ring of the blowup along the subschemes defined by $I_1,\ldots,I_r$ and (algebraically) the behaviour of all products that can be formed among the ideals $I_1,\ldots,I_r$.  The multi-Rees algebra of $I_1,\ldots,I_r$ is a special case of the Rees algebra of a module (see~\cite{SimisUlrichVasconcelos03} and~\cite{EisenbudHunekeUlrich03}).  Concretely, the multi-Rees algebra of the ideals $I_1,\ldots,I_r$ in a polynomial ring $R=\K[x_1,\ldots,x_n]$ is defined as:
\[
R[\calI\bt]=R[I_1t_1,\ldots,I_rt_r]:=\oplus_{a_1+\cdots+a_r\ge 0} I_1^{a_1}\cdots I_r^{a_r}t_1^{a_1}\cdots t_r^{a_r},
\]
where $\bt=\{t_1,\ldots,t_r\}$ are auxiliary variables.  A central problem for Rees and multi-Rees algebras is to describe their defining equations -- that is, to find a polynomial ring $S$ and an ideal $J=J_{R[\calI\bt]}\subset S$ so that $S/J\cong R[\calI\bt]$.  While this problem has been studied mostly for Rees algebras, there is a growing literature on the defining equations of multi-Rees algebras~\cite{R99,LP14,Sosa14,Jb18,BC17b,CLS19}.  

In this paper we study the Koszul property of the multi-Rees algebra of certain Borel ideals (these are also called \textit{strongly stable} ideals in the literature).  A graded ring $R\cong \bigoplus_{i=0}^\infty R_i$ over a field $\K=R_0$ with irrelevant ideal $R_+=\bigoplus_{i=1}^\infty R_i$ is Koszul if $R$ has a linear resolution over its residue field $R/R_+\cong \K$.  Koszul rings have good homological properties which closely mirror polynomial rings; see~\cite{CDR13} for a survey.  One way to establish that a ring is Koszul is to show that it is presented as a quotient $S/J$ of a polynomial ring $S$ where the ideal $J$ has a Gr\"obner basis of quadrics~\cite[Section~3.1]{CDR13}.  

Our main result in this paper, Theorem~\ref{thm:MultiSink}, is that the defining equations of the multi-Rees algebra of certain collections of what we call \textit{principal $L$-Borel} ideals have a Gr\"obner basis of quadrics with squarefree lead terms.  In~\cite[Page~3]{BC17}, Bruns and Conca write that it is `very likely' that the multi-fiber ring of the multi-Rees algebra of principal Borel ideals is defined by a Gr\"obner basis of quadrics.  This is a particular case of our main result -- see Corollary~\ref{cor:PrincipalBorel}.

Our proof of Theorem~\ref{thm:MultiSink} involves several ingredients.  The first is a fiber-wise criterion for a set of binomials to either generate or form a Gr\"obner basis for the kernel of a toric map.  Variations on this criterion first appear in~\cite{Blasiak08} and are further developed in~\cite{Schweig11,DFMSS19} for particular toric maps.  In Section~\ref{sec:ToricMaps} we give a careful statement and proof of this criterion for arbitrary toric maps, which we then use throughout the paper.  The second ingredient is a modification of Sturmfels' sorting algorithm~\cite[Chapter~14]{S96}.  In~\cite{DN99} De Negri shows that the resulting \textit{sorting order} yields a Gr\"obner basis of quadrics for the defining equations of the toric ring of a principal Borel ideal.  This argument is difficult to extend to multi-Rees algebras because the sorting order is sensitive to the ideal chosen; see~\cite{Sosa14} where Sosa uses the sorting order to establish the Koszul property for the multi-Rees algebra of certain principal Borel ideals.  We give a modified sorting algorithm in Section~\ref{sec:Principal} (Algorithm~\ref{alg:BS}) which produces the \textit{lexicographically} least monomial in each monomial fiber of the toric map associated to a principal Borel ideal.  Using the fiber-wise Gr\"obner basis criterion developed in Section~\ref{sec:ToricMaps}, we conclude that the defining equations of the toric ring of a principal Borel ideal has a Gr\"obner basis of quadrics with respect to \textit{lexicographic} order.  In Section~\ref{sec:MultiRees} we build on this to show that the multi-Rees algebra of certain collections of principal $L$-Borel ideals have a Gr\"obner basis of quadrics with respect to lexicographic order also.

We now explain what we mean by principal $L$-Borel ideals and give the idea for which collections of principal $L$-Borel ideals have a Koszul multi-Rees algebra.
We have taken the notation of an $L$-Borel ideal from~\cite{FMS13}, where $Q$-Borel ideals are introduced for a partially ordered set $Q$ on the underlying variables of the polynomial ring.  An ideal $I\subset R=\K[x_1,\ldots,x_n]$ is Borel (also \textit{strongly stable} in the literature) if, whenever $m\in I$ is a monomial and $x_j$ divides $m$, $\frac{x_i}{x_j}m\in I$ for any $i<j$.  We call $\frac{x_i}{x_j}m$ a \textit{Borel move} on $m$.  The ideal $I$ is \textit{principal Borel} if its generators can be obtained from Borel moves on a single monomial, which we call the Borel generator of $I$.  Now suppose that $L$ is a \textit{linear ordering} on a subset of $\{x_1,\ldots,x_n\}$ which respects the usual lexicographic ordering of the variables.  For instance, if $n=5$, we could write $x_1>_L x_3>_L x_5$, while $x_2$ and $x_4$ are not compared to any other variable by $L$.  Then a monomial ideal $I$ is $L$-Borel if, for any monomial $m\in I$ so that $x_j$ divides $m$, and any variable $x_i$ so that $x_i>_L x_j$, $\frac{x_i}{x_j}m\in I$.  In this case we call $\frac{x_i}{x_j}m$ an $L$-Borel move.  The ideal $I$ is a principal $L$-Borel ideal if its generators are obtained by $L$-Borel moves from a single monomial, which we call the $L$-Borel generator of $I$.  For example, if $L$ is the linear ordering on $\{x_1,\ldots,x_5\}$ considered above, $\langle x_1,x_3,x_5\rangle$ is a principal $L$-Borel ideal with $L$-Borel generator $x_5$.

Now suppose that we are given a collection of monomial ideals $\calI=\{I_1,\ldots,I_r\}$ so that $I_i$ is principal $L_i$-Borel with respect to a linear poset $L_i$ for $i=1,\ldots,r$. ($L_1,\ldots, L_r$ are not necessarily the same.)  We define a bipartite incidence graph $G(\calI)$ associated to $\calI$ as follows.  The vertices of $G(\calI)$ are the variables $x_1,\ldots,x_n$ of the polynomial ring and the auxiliary variables $t_1,\ldots,t_r$ of the multi-Rees algebra, and $t_i$ is connected by an edge to $x_j$ if and only if $x_j$ is comparable to another variable by the linear poset $L_j$ (our definition for $G(\calI)$ in Section~\ref{sec:LBorel} is slightly more nuanced, but this is sufficient for now).  This graph records which Borel moves are \textit{allowable} for each ideal of $\calI$.  In essence, our main result (Theorem~\ref{thm:MultiSink}) is that the multi-Rees algebra $R[\calI\bt]$ is Koszul if the graph $G(\calI)$ is chordal bipartite (the actual criterion is slightly more complicated -- we direct the reader to Sections~\ref{sec:LBorel} and~\ref{sec:MultiRees} for the details).  More precisely, Theorem~\ref{thm:MultiSink} shows that Algorithm~\ref{alg:BS} can be used successively in a greedy fashion (according to an appropriate ordering of the vertices of the incidence graph $G(\mathcal{I})$) to produce the lexicographically least monomial in the fiber over a monomial in the toric map naturally associated to the multi-Rees algebra of $\calI$.  Our main result is inspired by (and significantly extends) the result of Ohsugi and Hibi that the toric ring of the edge ideal of a graph is Koszul if the graph is chordal bipartite~\cite{OH99}.


The paper is organized as follows.  In Section~\ref{sec:ToricMaps} we formalize a criterion originating in~\cite{Blasiak08} and appearing in~\cite{Schweig11,DFMSS19} for a set of binomials to be a generating set (respectively, Gr\"obner basis) for the kernel of a toric map.  We then describe the multi-Rees algebra of monomial ideals as a toric map.  In Section~\ref{sec:Borel} we recall and prove some relevant properties of principal Borel ideals.  We then proceed in Section~\ref{sec:Principal} to describe the $\bs$ algorithm and prove that it produces squarefree quadratic leading terms in lexicographic order, recovering De Negri's result~\cite{DN99} using lexicographic order instead of sorting order.  In Section~\ref{sec:LBorel} we introduce principal $L$-Borel ideals and define the notion of an $\mathsf{L}$-free ordering of a collection of principal $L$-Borel ideals (this is one way to encode the chordal bipartite property of the incidence graph).  In Section~\ref{sec:MultiRees} we prove our main result - that an $\mathsf{L}$-free collection of principal $L$-Borel ideals has a Koszul multi-Rees algebra.  We conclude in Section~\ref{sec:Conclusion} with some related remarks and questions for further research.  We close this introduction with a simple example illustrating our main result.

\begin{example}\label{ex:intro}
	Consider the polynomial ring $R=\K[x_1,x_2,x_3,x_4]$.  Let $L_1=L_2$ be the linear poset ordering $\{x_3,x_4\}$ with respect to decreasing subscripts (so $x_3>_{L_1} x_4$) and let $L_3$ and $L_4$ be the linear posets likewise ordering $\{x_2,x_3,x_4\}$ and $\{x_1,x_2,x_3\}$.  Put $\calI=\{I_1,I_2,I_3,I_4\}$ where $I_1=\langle x_3^3,x_3^2x_4,x_3x_4^2\rangle$, $I_2=\langle x_3^2,x_3x_4\rangle$, $I_3=\langle x_2^3,x_2^2x_3,x_2^2x_4,x_2x_3^2,x_2x_3x_4\rangle$, and $I_4=\langle x_1^2,x_1x_2,x_1x_3,x_2^2,x_2x_3,x_3^2\rangle$.  Then $I_i$ is principal $L_i$-Borel for $i=1,2,3,4$ with principal $L_i$-Borel generators $x_3x_4^2$, $x_3x_4$, $x_2x_3x_4$, and $x_3^2$, respectively.  The bipartite incidence graph $G(\calI)$ is shown in Figure~\ref{fig:bipartiteincidencegraph}.  The bi-adjacency matrix of $G(\calI)$ is
	\[
	\kbordermatrix{ & t_1 & t_2 & t_3 & t_4 \\
		x_1 & 0 & 0 & 0 & 1\\ 
		x_2 & 0 & 0 & 1 & 1\\
		x_3 & 1 & 1 & 1 & 1\\
		x_4 & 1 & 1 & 1 & 0
	}
	\]
	Since the bi-adjacency matrix has no sub-matrix of the form $\left[ \begin{array}{cc} 1 & 0\\ 1 & 1\end{array} \right]$ (we will say $\calI$ is $\mathsf{L}$-free in Section~\ref{sec:MultiRees}), it follows that $G(\calI)$ is chordal bipartite (see Theorem~\ref{thm:ChordalBipartite}).  Hence Theorem~\ref{thm:MultiSink} implies that the multi-Rees algebra $R[\calI\bt]$ is Koszul.
	\begin{figure}
		\centering
		\begin{tikzpicture}
		\node (x1) at (0,4) {$x_1$};
		\node (x2) at (0,3) {$x_2$};
		\node (x3) at (0,2) {$x_3$};
		\node (x4) at (0,1) {$x_4$};
		
		\node (t1) at (2,4) {$t_1$};
		\node (t2) at (2,3) {$t_2$};
		\node (t3) at (2,2) {$t_3$};
		\node (t4) at (2,1) {$t_4$};
		
		\draw (t1)--(x3) (t1)--(x4) (t2)--(x3) (t2)--(x4) (t3)--(x2) (t3)--(x3) (t3)--(x4) (t4)--(x1) (t4)--(x2) (t4)--(x3);
		\end{tikzpicture}
		\caption{The bipartite incidence graph for Example~\ref{ex:intro}}
		\label{fig:bipartiteincidencegraph}
	\end{figure}
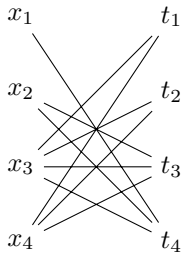
\end{example}

\section{Toric maps and toric fiber graphs}\label{sec:ToricMaps}
In this section we develop a Gr\"obner basis criterion for the kernel of a toric map which has appeared in~\cite{Blasiak08,Schweig11,DFMSS19} for particular choices of toric map.  Throughout we write $R$ for the polynomial ring $\K[x_1,\ldots,x_n]$ over a field $\K$.  If $I$ is a monomial ideal, we write $\gens(I)$ for the unique set of minimal generators of $I$.  Standard references for toric maps and toric ideals are~\cite{MS05} and~\cite{HHO18}.

\begin{definition}
	Let $G$ be a finite collection of monomials of positive degree in the polynomial ring $R$.  We write $\K[G]$ for the subring of $R$ generated by the monomials of $G$, which we call the \textit{toric ring} of $G$.  Let $S$ be the polynomial ring $S:=\K[T_m:m\in G]$ with an indeterminate for every monomial of $G$.  The \textit{toric map} associated to $G$ is the ring map $\phi_G: S\to R$ defined on variables by $\phi_G(T_m)=m$.  The image of $\phi_G$ is clearly $\K[G]$.  We write $J_{\K[G]}$ for the kernel of $\phi_G$; we call this the \textit{toric ideal} of $G$.  This is the defining ideal of the subring $\K[G]$ presented as a quotient of $S$.
	
	If $I$ is a monomial ideal of $R$ we write $\K[I]$ for $\K[\gens(I)]$; we call $\K[I]$ the toric ring of $I$.  In this case we write $\phi_I$ instead of $\phi_{\gens(I)}$ for the toric map.  Likewise we write $J_{\K[I]}$ for the kernel of $\phi_I$; we also call $J_{\K[I]}$ the toric ideal of $I$.
\end{definition}

\begin{remark}
	For toric ideals it is standard to replace $R$ by the Laurent polynomial ring and allow $G$ to have monomials with negative exponents (see~\cite[Chapter~3]{HHO18}).  We will not consider toric ideals in this generality.  Alternatively, we consider only toric ideals of affine monoids whose convex hull is \textit{pointed}.  See~\cite[Chapter~8]{MS05} for more details.
\end{remark}

Let $G$ be a finite collection of monomials of positive degree in the polynomial ring $R$.  Given $\bgamma=(\gamma_m)\in\Z_{\ge 0}^{G}$, we write $\bT^{\bgamma}$ for $\prod_{m\in G}T_m^{\gamma_m}$, where $T_m$ is a variable of the polynomial ring $S=\K[T_m:m\in G]$.  We put a multigrading on $S$ by the monomials of $\K[G]$ as follows.  If $\mu\in \K[G]$, then
\[
S_\mu=\mbox{span}_\K\{\bT^{\bgamma}: \phi_G(\bT^{\bgamma})=\mu\}.
\]

We say an $S$-module $M$ is \textit{graded} by $\K[G]$ if $M$ decomposes as a direct sum of $\K$-vector spaces as
\[
M=\bigoplus_{\mu\in\K[G]} M_\mu
\]
and, for any $\mu,\nu\in \K[G]$, if $m\in M_\mu$ and $f\in S_{\nu}$, then $f\cdot m\in M_{\mu\nu}$ (we write this as a product since $\mu$ and $\nu$ are monomials).  It is well-known that the toric ideal $J=J_{K[G]}$ is graded in this way (see~\cite[Chapter~8]{MS05}), and moreover we can precisely identify the graded pieces of $J$: if $\mu\in \K[G]$ is a monomial, then
\[
\begin{array}{rl}
J_\mu & =\mbox{span}_\K\{\bT^{\bgamma_1}-\bT^{\bgamma_2}: \bT^{\bgamma_1},\bT^{\bgamma_2}\in S_\mu\}\\[10 pt]
& =\mbox{span}_\K\{\bT^{\bgamma_1}-\bT^{\bgamma_2}: \phi_G(\bT^{\bgamma_1})=\phi_G(\bT^{\bgamma_2})=\mu\}.
\end{array}
\]
That is, $J$ is a binomial prime ideal, also known as a toric ideal.  See~\cite[Chapter~7]{MS05} or~\cite[Chapter~3]{HHO18} for additional details.

We now introduce a combinatorial device from~\cite{Blasiak08} and~\cite{Schweig11} (see also~\cite{DFMSS19}), which we call the \textit{fiber graph} of the toric map $\phi_G$ at the monomial $\mu$.

\begin{definition}\label{def:FiberGraph}
	Let $G$ be a finite collection of monomials of positive degree in the polynomial ring $R$, $J_{\K[G]}$ the toric ideal of $G$, and $\calB\subset J_{\K[G]}$ a finite collection of binomials.  The \textit{fiber graph of $\phi_G$ at $\mu$ with respect to $\calB$} is the graph $\Gamma_{\mu,\calB}$ whose vertices are monomials $\bT^{\bgamma}\in S_\mu$ with an edge connecting $\bT^{\bgamma},\bT^{\bgamma'}\in S_\mu$ if $\bT^{\bgamma}-\bT^{\bgamma'}$ is divisible by a binomial from $\calB$.
	
	If, moreover, $\prec$ is a monomial order on $S$, then $\vec{\Gamma}_{\mu,\calB}$ is the graph $\Gamma_{\mu,\calB}$ with edges directed from the \textit{larger} monomial to the \textit{smaller}.  That is, if $\bT^{\bgamma},\bT^{\bgamma'}\in S_\mu$ are connected by an edge in $\Gamma_{\mu,\calB}$ and $\bT^{\bgamma'}\prec\bT^{\bgamma}$, then we get the directed edge $\bT^{\bgamma}\to \bT^{\bgamma'}$.
\end{definition}

\begin{remark}
	We suppress the collection $G$ of monomials of $R$ in the notation for $\Gamma_{\mu,\calB}$, assuming that the underlying toric map is understood from context.
\end{remark}

The significance of these fiber graphs comes from the following proposition.  See also~\cite[Proposition~4.5]{DFMSS19}.

\begin{proposition}\label{prop:Graph}
	Let $\phi_G:S\to R$ be a toric map and $\calB$ a collection of binomials from the toric ideal $J=J_{\K[G]}$.
	The following are equivalent:
	\begin{enumerate}
		\item The binomials in $\calB$ generate $J$.
		\item The graph $\Gamma_{\mu,\calB}$ is either empty or connected for every $\mu\in R$.
	\end{enumerate}
	If $S$ is equipped with a monomial order $\prec$, then the following are equivalent:
	\begin{enumerate}
		\item The binomials in $\calB$ form a Gr\"obner basis for $J$ under $\prec$.
		\item The directed graph $\vec{\Gamma}_{\mu,\calB}$ is either empty or has a unique sink for every $\mu\in R$.
	\end{enumerate}
\end{proposition}
\begin{proof}
	First, assume every nonempty graph $\Gamma_{\mu,\calB}$ is connected.  Since $J$ is graded by $\K[G]$, it suffices to show that every binomial of the form $\bT^{\bgamma}-\bT^{\bgamma'}$, where $\bT^{\bgamma},\bT^{\bgamma'}\in S_\mu$, can be written as a combination of binomials from $\calB$.  Since $\Gamma_{\mu,\calB}$ is connected, there is a path $P$ from $\bT^{\bgamma}$ to $\bT^{\bgamma'}$ in $\Gamma_{\mu,\calB}$.  That is,
	\[
	\bT^{\bgamma}-\bT^{\bgamma'}=\sum_{i=1}^k (\bT^{\bgamma_i}-\bT^{\bgamma'_i}),
	\]
	where $\bT^{\bgamma_i}$ and $\bT^{\bgamma'_i}$ are endpoints of the $i^{\mbox{\footnotesize{th}}}$ edge in this path (thus $\bT^{\bgamma'_i}=\bT^{\bgamma_{i+1}}$ for $i=1,\ldots,k-1$).  By definition of $\Gamma_{\mu,\calB}$, $\bT^{\bgamma_i}-\bT^{\bgamma'_i}=\epsilon_i m_i(A_i-B_i)$ for some $\epsilon_i\in\{-1,1\}$, some binomial $A_i-B_i\in\calB$, and some monomial $m_i\in R$.  So $\bT^{\bgamma}-\bT^{\bgamma'}=\sum_{i=1}^k \epsilon_i m_i(A_i-B_i)$, as desired.
	
	Now assume that the set of binomials $\calB$ generates $J$.  Let $\bT^{\bgamma}-\bT^{\bgamma'}\in J_\mu$, so $\bT^{\bgamma},\bT^{\bgamma'}$ are vertices of $\Gamma_{\mu,\calB}$. Then, by~\cite[Lemma~3.8]{HHO18},
	\begin{equation}\label{eq:binomexp}
	\bT^{\bgamma}-\bT^{\bgamma'}=\sum_{i=1}^k \epsilon_i m_i(A_i-B_i),
	\end{equation}
	where $\epsilon_i\in\{-1,1\}$, $m_i$ is a monomial in $S$, $A_i-B_i\in\calB$ for $1\le i\le k$, and $m_i(A_i-B_i)\neq m_j(A_j-B_j) \mbox{ for any } 1\le i,j\le k$.  It is clear that $m_iA_i$, $m_iB_i$ are adjacent vertices of $\Gamma_{\mu,\calB}$ for $i=1,\ldots,k$.  Thus each term in the sum on the left-hand side of~\eqref{eq:binomexp} yields an edge of $\Gamma_{\calB,\mu}$, and no edge is repeated since $m_i(A_i-B_i)\neq m_j(A_j-B_j) \mbox{ for any } 1\le i,j\le k$.  We consider the subgraph $H$ consisting of all of these edges.  Since all of the monomials $m_iA_i,m_iB_i$ must cancel in~\eqref{eq:binomexp}, every vertex of $H$ except $\bT^{\bgamma}$ and $\bT^{\bgamma'}$ must have even degree.  Since $\bT^{\gamma}$ and $\bT^{\bgamma'}$ are the only vertices of $H$ with odd degree, they must be in the same connected component of $H$.  Thus $\Gamma_{\mu,\calB}$ is connected.
	
	Now suppose that $S$ is equipped with a monomial order $\prec$ and $\vec{\Gamma}_{\mu,\calB}$ has a unique sink for every $\mu$ (whenever $\vec{\Gamma}_{\mu,\calB}$ is non-empty).  We claim $\calB$ is a Gr\"obner basis for $J$.  Fix a monomial $\mu\in R$.  Suppose $\bT^{\gamma'}$ is the unique sink of $\Gamma_{\mu,\calB}$.  We first claim every monomial $\bT^{\bgamma}\neq \bT^{\bgamma'}$ is connected to $\bT^{\gamma'}$ by a directed path.  To do this, we inductively build a path starting at $\bT^{\bgamma}=\bT^{\bgamma_1}$ as follows.  Since $\bT^{\bgamma}$ is not a sink, we can choose an outgoing edge of $\bT^{\bgamma}$; call the terminal vertex of this edge $\bT^{\bgamma_2}$.  If $\bT^{\bgamma_2}=\bT^{\bgamma'}$, we are done.  Otherwise, $\bT^{\bgamma_2}$ is not a sink, so we pick an outgoing edge of $\bT^{\bgamma_2}$ and call the terminal vertex of this edge $\bT^{\bgamma_3}$.  Continuing in this way we will either terminate at the sink $\bT^{\bgamma'}$ or necessarily repeat a vertex at some point.  If we repeat a vertex, we get a directed cycle, which is impossible since $\prec$ is a monomial order (hence $\vec{\Gamma}_{\mu,\calB}$ is directed acyclic).  So this process must result in a path terminating at the unique sink.  This also implies that the unique sink $\bT^{\bgamma'}$ is the smallest monomial in $S_\mu$ under $\prec$.
	
	Now, to prove that $\calB$ is a Gr\"obner basis of $J$, it suffices to prove that the lead term of every polynomial $f\in J$ is divisible by the lead term of some binomial in $\calB$.  Since $J$ is graded, it suffices to prove that the lead term of every binomial $\bT^{\bgamma}-\bT^{\gamma'}\in J_\mu$ is divisible by the lead term of some binomial in $\calB$.  From the definition of $\vec{\Gamma}_{\mu,\calB}$, the initial vertex of every directed edge corresponds to a lead term of a binomial.  The condition that $\vec{\Gamma}_{\mu,\calB}$ has a unique sink means that every monomial $\bT^{\bgamma}\in\vec{\Gamma}_{\mu,\calB}$ which is different from the unique sink is a lead term of some binomial.  Also, the argument in the previous paragraph shows that the unique sink will never be the lead term of any binomial.  In other words, the lead term ideal of $J$ in degree $\mu$ consists of every vertex of $\vec{\Gamma}_{\mu,\calB}$ besides the unique sink.  Now we show that every such monomial is divisible by the lead term of some binomial in $\calB$.  So suppose that $\bT^{\bgamma}\in \vec{\Gamma}_{\mu,\calB}$ is not the unique sink of $\vec{\Gamma}_{\mu,\calB}$.  Since it is not the unique sink, there is a directed edge $\bT^{\bgamma}\to\bT^{\bgamma'}$ in $\vec{\Gamma}_{\mu,\calB}$, hence $\bT^{\bgamma'}\prec \bT^{\bgamma}$ and $\bT^{\bgamma}-\bT^{\bgamma'}=m(A-B)\in J_\mu$ for some binomial $A-B\in\calB$ and monomial $m\in S$, where we assume $B\prec A$.  Since $B\prec A$ and $\prec$ is a monomial order, $mB\prec mA$, so $mA=\bT^{\bgamma}$.  So $\bT^{\bgamma}$ is divisible by the lead term of a binomial in $\calB$, proving the claim.
	
	Finally, suppose that the binomials in $\calB$ form a Gr\"obner basis for $J$.  Suppose for a contradiction that for some $\mu$, $\vec{\Gamma}_{\mu,\calB}$ is non-empty and does not have a unique sink.  That is, suppose that $\bT^{\bgamma},\bT^{\bgamma'}\in R_\mu$ and both are sinks in $\vec{\Gamma}_{\mu,\calB}$.  Since both are sinks, it follows that neither of $\bT^{\bgamma},\bT^{\bgamma'}$ is divisible by the lead term of any binomial in $\calB$.  But then the binomial $\bT^{\bgamma}-\bT^{\bgamma'}\in J_\mu$ and its lead term is not divisible by any binomial of $\calB$, contradicting that $\calB$ is a Gr\"obner basis of $J$.
\end{proof}

\subsection{Multi-Rees algebras}\label{ss:multirees}
In this section we review the toric maps which are the main focus of this paper.  A robust development of the details this section may be found in~\cite{CLS19}, so we will be brief.
Let $\K$ be a field.  Write $\x$ for the set of variables $\{x_1,\ldots,x_n\}$ and $\K[\x]$ for the polynomial ring $\K[x_1,\ldots,x_n]$.  If $\balpha=(\alpha_1,\ldots,\alpha_n)\in\Z_{\ge 0}^n$, we write $\x^{\balpha}$ for $x_1^{\alpha_1}\cdots x_n^{\alpha_n}$.

Let $\calI=\{I_1,\ldots,I_r\}$ be monomial ideals in $\K[\x]$.  The multi-Rees algebra of $\calI$ is the following subring of $\K[\x][t_1,\ldots,t_r]$:
\[
R[\calI\bt]=R[I_1t_1,\ldots,I_rt_r]=\bigoplus\limits_{a_1,\ldots,a_k\ge 0} I_1^{a_1}\cdots I_r^{a_r}t_1^{a_1}\cdots t_r^{a_r}.
\]
In case $\calI=\{I\}$, a single ideal, then we write $R[It]$ instead of $R[\calI\bt]$; this is the Rees algebra of $I$.

In keeping with our notation for the variables $x_1,\ldots,x_n$, we write $\bt$ for the set of variables $\{t_1,\ldots,t_r\}$ and if $\bbeta\in\Z_{\ge 0}^k$ then we write $\bt^{\bbeta}$ for the monomial $t_1^{\beta_1}\cdots t_r^{\beta_r}$.  We also write $\K[\x,\bt]$ for the polynomial ring $\K[\x][t_1,\ldots,t_r]=\K[x_1,\ldots,x_n,t_1,\ldots,t_r]$.

Let $G_1,\ldots,G_r$ be minimal sets of generators for $I_1,\ldots,I_r$. Then clearly
\[
R[\calI\bt]=\K[x_1,\ldots,x_n,\{\x^{\balpha} t_k: \x^{\balpha}\in G_k,1\le k\le r\}]
\]
We create a variable $T_{\x^\alpha t_k}$ for each monomial $\x^{\balpha}t_k$, where $\x^{\balpha}\in G_k$ and $1\le k\le r$, and write $\bT$ for the set of all such variables.  We then present the multi-Rees algebra $R[\calI\bt]$ by the map
\[
\phi:\K[\x,\bT]\to \K[\x,\bt]
\]
defined as $\phi(x_i)=x_i$ for all $x_i\in\x$ and $\phi(T_{\x^{\balpha}t_i})=x^{\balpha}t_i$ for all $\x^{\balpha}t_i$ with $\x^{\balpha}\in G_i$.  Clearly this is a toric map.  We will be concerned with the \textit{defining equations} of $R[\calI\bt]$, which is the toric ideal $\ker(\phi)=J_{R[\calI\bt]}$.  We also call this the \textit{multi-Rees} ideal of $\calI$, or simply the \textit{Rees ideal} of $I$ if $\calI$ consists of the single ideal $I$.


Closely related to the multi-Rees algebra is the \textit{multi-fiber} ring of $R[\calI\bt]$.  Write $\mathfrak{m}$ for the ideal $\langle x_1,\ldots,x_n\rangle\subset \K[\x]$.  Then the multi-fiber ring of $R[\calI\bt]$ is $R[\calI\bt]/\mathfrak{m}R[\calI\bt]$ (viewing $R[\calI\bt]$ as a $\K[\x]$-module).  If the monomial ideals $\calI=\{I_1,\ldots,I_k\}$ are each generated in a single degree, then we have an isomorphism
\[
R[\calI\bt]/\mathfrak{m}R[\calI\bt]\cong \K[\x^{\balpha} t_k: \x^{\balpha}\in G_k, 1\le k\le r].
\]
We denote the ring $\K[\x^{\balpha} t_k: \x^{\balpha}\in G_k, 1\le k\le r]$ by $\K[\calI\bt]$.  If $\calI=\{I\}$ consists of a monomial ideal generated in a single degree, then
\[
R[It]/\mathfrak{m}R[It]\cong \K[\calI\bt]\cong \K[I],
\]
the toric ring of $I$.

\section{Borel Ideals}\label{sec:Borel}

In this section we collect some notation and results regarding Borel moves and Borel ideals.  Most of these can be found in the literature (c.f.~\cite{DN99,DFMSS19}).  What we refer to as a \textit{Borel ideal} is often called a \textit{strongly stable} ideal.  A Borel ideal is the same as a \textit{Borel-fixed} ideal in characteristic $0$, although this correspondence breaks down in higher characteristic.  Throughout we consider the polynomial ring $\K[\x]=\K[x_1,\ldots,x_n]$.

\begin{definition}\label{def:BorelDefinitions}
	If $m$ is a monomial in $\K[\x]$ which is divisible by $x_j$ and $i<j$, a \textit{Borel move} (on $m$) consists of replacing $m$ by $\frac{x_i}{x_j}m$.  If instead $m$ is divisible by $x_i$, a \textit{reverse Borel move} (on $m$) consists of replacing $m$ by $\frac{x_j}{x_i}m$.  If $\mathcal{S}\subset \K[\x]$ is a collection of monomials, we say it is \textit{Borel closed} if any Borel move on a monomial in $\mathcal{S}$ is also in $\mathcal{S}$.  If $\mathcal{S}$ is not Borel closed, we write $\Borel(\mathcal{S})$ for the smallest Borel closed set of monomials containing $\mathcal{S}$.  Clearly $\Borel(\mathcal{S})$ consists of all monomials which can be obtained from monomials of $\mathcal{S}$ via Borel moves.  If $\mathcal{S}$ is Borel closed, we call $\mathcal{S}$ a \textit{principal Borel set} if $\mathcal{S}=\Borel(M)$ for a single monomial $M$.  If $I$ is a monomial ideal, we call it \textit{principal Borel} if it is generated by a principal Borel set.
\end{definition}

\begin{remark}
	See~\cite{FMS11} for more on the Borel generators of a Borel ideal, and the wealth of information that can be obtained from this set.
\end{remark}

\begin{definition}\label{def:BorelOrder}
	We define the \textit{Borel (partial) order}, denoted $<_{\Borel}$, on the monomials of $\K[\x]$ by $m<_{\Borel}m'$ if $m'$ can be obtained from $m$ by a sequence of Borel moves.  \textit{Any} monomial order $\prec$ on $\K[\x]$ which satisfies $x_1\succ x_2\succ \cdots \succ x_n$ is a refinement of the Borel order.  In other words, if $m <_{\Borel}m'$, then $m\prec m'$.
\end{definition}

\begin{remark}
	We caution the reader that the Borel partial order is often given in the literature as the reverse of how we have presented it in Definition~\ref{def:BorelOrder} (c.f~\cite{FMS13,DFMSS19}).  We have made the choice in Definition~\ref{def:BorelOrder} so that the Borel partial order is compatible with monomial orders.
\end{remark}

\begin{remark}
	We will denote by $\prec_{\lex},\prec_{\glex},$ and $\prec_{\grevlex}$ the \textit{lexicographic},\textit{graded lexicographic}, and \textit{graded reverse lexicographic} monomial orders.
\end{remark}

\begin{example}
	We illustrate how the graded reverse lexicographic and lexicographic monomial orders on $\K[x_1,x_2,x_3,x_4]$ refine the Borel order. 
	First, if monomials are comparable in the Borel order then they are compared in the same way by any monomial order.  For example, $x_1x_2x_3<_{\Borel}x_1^2x_2$.  Notice that $x_1x_2x_3\prec_{\grevlex}x_1^2x_2$, and $x_1x_2x_3\prec_{\lex} x_1^2x_2$.  On the other hand, monomials which are incomparable under the Borel order may be compared in different ways by different monomial orders.  For instance, the monomials $x_2^2x_3$ and $x_1^2x_4$ are incomparable in the Borel order since neither can be obtained from the other by Borel moves.  Notice that $x_1^2x_4\prec_{\grevlex}x_2^2x_3$ while $x_2^2x_3\prec_{\lex} x_1^2x_4$.
\end{example}

We will use a few more results about principal Borel ideals.

\begin{definition}\label{def:CumulativeExponents}
	Suppose $m=x_1^{a_1}\cdots x_n^{a_n}\in \K[\x]$.  We define
	\[
	\sigma_i(m)=a_i+a_{i+1}+\cdots+a_n.
	\]
\end{definition}

\begin{lemma}\label{lem:BorelMembership}
	Suppose $m,M$ are monomials of the same degree in $\K[\x]$.  Then $m\in \Borel(M)$ if and only if $\sigma_i(m)\le \sigma_i(M)$ for $i=1,\ldots,n$.
\end{lemma}
\begin{proof}
	This is part of~\cite[Lemma~1.3]{DN99}.
\end{proof}

\begin{observation}\label{lem:ReverseBorelMoveCriterion}
	If $m\in \Borel(M)$ and $\sigma_j(m)<\sigma_j(M)$, then there exists an index $i<j$ so that $\frac{x_j}{x_i}m\in\Borel(M)$.  Equivalently, if $m\in\Borel(M)$ and $m\neq M$, then there is a reverse Borel move on $m$ that is also in $\Borel(M)$.
\end{observation}

Since $<_{\Borel}$ is a partial order, an arbitrary set of monomials of the same degree may not have a unique maximal or minimal element under the Borel order.  The following lemma, which is a variation on~\cite[Lemma~5.1]{DFMSS19}, gives one instance where unique minimal elements exist under the Borel order.  Due to its centrality in our arguments we give a detailed proof.

\begin{lemma}\label{lem:BorelMultidegreeDivision}
	Let $M,\mu\in \K[\x]$ be monomials.  If there is a monomial in $\Borel(M)$ dividing $\mu$, then there is a unique minimal monomial $M'$ under the Borel partial order so that $M'\in\Borel(M)$ and $M'$ divides $\mu$.  In other words, if $m$ is any monomial in $\Borel(M)$ which divides $\mu$ then $m\in\Borel(M')$.  Moreover, suppose $m\neq M'$ and let $j$ be the maximal index so that $\sigma_j(m)<\sigma_j(M')$.  Then there is an index $i<j$ so that $\frac{x_j}{x_i}m\in \Borel(M)$ and $\frac{x_j}{x_i}m$ divides $\mu$.
\end{lemma}
\begin{proof}
	Assume that the set of monomials in $\Borel(M)$ which divide $\mu$ is non-empty.  Let $M'$ be a minimal element of this set under Borel order.  Now suppose that $m\in\Borel(M)$ and $m$ divides $\mu$.  For a contradiction, suppose that $m\notin\Borel(M')$.
	
	By Lemma~\ref{lem:BorelMembership} there is a maximal index $1\le j\le n$ so that $\sigma_j(m)>\sigma_j(M')$.  Since $\sigma_j(M)\ge \sigma_j(m)>\sigma_j(M')$, by Observation~\ref{lem:ReverseBorelMoveCriterion} there is an index $i<j$ so that $M''=\frac{x_j}{x_i}M'\in \Borel(M)$.  Clearly $M''<_{Borel} M'$.  We show that $M''$ divides $\mu$.  Since $\sigma_j(m)>\sigma_j(M')$, and this is the maximal index for which $\sigma_k(m)$ is greater than $\sigma_k(M')$, it follows that the exponent of $x_j$ in $m$ is strictly greater than it is in $M'$.  Since $m$ divides $\mu$, it follows that the exponent of $x_j$ in $\mu$ is also strictly greater than the exponent of $x_j$ in $M'$.  Thus $x_jM'$ divides $\mu$, so clearly $M''$ divides $\mu$ as well.  Hence $M'$ is not a minimal monomial under Borel order in $\Borel(M)$ dividing $\mu$, a contradiction.
	
	Now suppose that $m\in\Borel(M)$, $m$ divides $\mu$, and $m\neq M'$, where $M'$ is the unique minimal monomial in $\Borel(M)$ under Borel order which divides $\mu$ (by the above paragraph, this monomial exists).  Since $m\in\Borel(M')$ and $m\neq M'$, by Lemma~\ref{lem:BorelMembership}, there is a maximal index $1\le j\le n$ so that $\sigma_j(m)<\sigma_j(M')$.  Thus by Observation~\ref{lem:ReverseBorelMoveCriterion}, there is an index $i<j$ so that $\frac{x_j}{x_i}m\in\Borel(M')$.  We show that $\frac{x_j}{x_i}m$ divides $\mu$.  Since $\sigma_j(m)<\sigma_j(M')$, and this is the maximal index for which $\sigma_k(m)$ is less than $\sigma_k(M')$, it follows that the exponent of $x_j$ in $M'$ is strictly greater than it is in $m$.  Since $M'$ divides $\mu$, it follows that the exponent of $x_j$ in $\mu$ is also strictly greater than the exponent of $x_j$ in $m$.  Thus $x_jm$ divides $\mu$, so clearly $\frac{x_j}{x_i}m$ divides $\mu$ as well.
\end{proof}

\begin{example}
	Consider the monomial $M=x_2^2x_4$ in the polynomial ring\\ $\K[x_1,x_2,x_3,x_4]$.  Then
	\[
	\Borel(M)=\{x_1^3, x_1^2x_2, x_1x_2^2, x_2^3, x_1^2x_3, x_1x_2x_3, x_2^2x_3,
	x_1^2x_4, x_1x_2x_4, x_2^2x_4\}.
	\]
	If $\mu=x_3x_4^2$ there is no monomial in $\Borel(M)$ dividing $\mu$.  On the other hand, if $\mu=x_1^2x_2x_3$, the monomials in $\Borel(x_2^2x_4)$ which divide $\mu$ are $x_1^2x_2$, $x_1^2x_3$, and $x_1x_2x_3$.  Clearly 
	$M'=x_1x_2x_3$ is the unique minimal monomial under Borel partial order among those which divide $\mu$.  Furthermore $x_1^2x_2$ and $x_1^2x_3$ are both in $\Borel(x_1x_2x_3)$, and each has a reverse Borel move transforming it into $x_1x_2x_3$.
\end{example}

Given a monomial $M$, we use the following lemma to determine if a monomial $\mu\in \K[\x]$ has a factorization whose factors belong to $\Borel(M)$.

\begin{lemma}\label{lem:BorelFactorization}
Let $M,N\in \K[\x]$.  Then $N$ factors as $N=M_1\cdots M_k$, where $M_1,\ldots,M_k\in\Borel(M)$, if and only if $N\in\Borel(M^k)$.
\end{lemma}
\begin{proof}
	First suppose that $N=M_1\ldots M_k$, where $M_1,\ldots,M_k\in\Borel(M)$.  Then $M_1,\ldots,M_k$ can each be obtained from $M$ by a sequence of Borel moves.  Applying the same Borel moves to each factor in the $k$-fold product $M^k=M\cdot M\cdots M$ yields that $N\in\Borel(M^k)$.  
	
	Now suppose that $N\in\Borel(M^k)$.  We show that $N=M_1\cdots M_k$ for some $M_1,\ldots,M_k\in\Borel(M)$ by induction on the number of Borel moves necessary to obtain $N$ from $M^k$.  For the base case, if $N=M^k$ then we can simply take $M_1=\cdots=M_k=M$.  Now suppose that $N$ can be obtained from $M^k$ in $\ell\ge 1$ Borel moves.  Then there is some $N'\in\Borel(M^k)$ so that $N=\frac{x_i}{x_j}N'$ (that is, $N$ is obtained from $N'$ by a Borel move), and $N'$ can be obtained from $M^k$ in $\ell-1$ Borel moves.  By induction, $N'=M'_1\cdots M'_k$ for some $M'_1,\ldots,M'_k\in\Borel(M)$.  Since $x_j$ divides $N'$, $x_j$ must divide one of the factors $M'_1,\ldots,M'_k$, without loss of generality suppose $x_j$ divides $M'_1$.  Put $M_1=\frac{x_i}{x_j}M'_1$, $M_2=M'_2,\ldots, M_k=M'_k$.  Clearly $M_1,\ldots,M_k\in\Borel(M)$ and $\mu=M_1\cdots M_k$.  This completes the induction.
\end{proof}

We will use the following refinement of Lemma~\ref{lem:BorelMultidegreeDivision} for factorizations of monomials in $\Borel(M^k)$.

\begin{lemma}\label{lem:BorelFactorizationMultidegreeDivision}
Let $M,\mu\in \K[\x]$ be monomials and $k\ge 1$ an integer.  If there is a monomial in $\Borel(M^k)$ dividing $\mu$, then let $P$ be the minimal monomial under Borel order in $\Borel(M^k)$ dividing $\mu$ guaranteed by Lemma~\ref{lem:BorelMultidegreeDivision}.  Suppose $P'\in\Borel(M^k)$ factors as $P'=P'_1\cdots P'_k$.  If $P'\neq P$, then there are indices $1\le i<j\le n$ and $1\le \ell\le k$ so that
\begin{enumerate}
\item $\frac{x_j}{x_i}P'_\ell\in \Borel(M)$
\item $\frac{x_j}{x_i}P'$ divides $\mu$
\item $\frac{x_j}{x_i}P'\in\Borel(M^k)$
\end{enumerate}
\end{lemma}
\begin{proof}
Let $j$ ($1\le j\le n$) be the maximal index so that $\sigma_j(P')<\sigma_j(P)$.  Since $P\in\Borel(M^k)$, $\sigma_j(P)\le k\sigma_j(M)$.  Since $\sigma_j(P'_1)+\cdots+\sigma_j(P'_k)=\sigma_j(P')<k\sigma_j(M)$, there exists an index $\ell$ so that $\sigma_j(P'_\ell)<\sigma_j(M)$.  Thus by Observation~\ref{lem:ReverseBorelMoveCriterion} there is some index $i<j$ so that $\frac{x_j}{x_i}P'_\ell\in\Borel(M)$.  Furthermore, since $j$ is the largest index so that $\sigma_j(P')<\sigma_j(P)$, the exponent of $x_j$ in $P'$ is strictly less than the exponent of $x_j$ in $P$ and hence also strictly less than the exponent of $x_j$ in $\mu$, since $P$ divides $\mu$.  As $P'$ also divides $\mu$ it follows that $x_jP'$ divides $\mu$, so clearly $\frac{x_j}{x_i}P'$ divides $\mu$.  Furthermore, $\frac{x_j}{x_i}P'$ factors as $(\frac{x_j}{x_i}P'_\ell)\prod_{t\neq \ell} P'_t$; since all of these factors are in $\Borel(M)$, it follows that $\frac{x_j}{x_i}P'\in\Borel(M^k)$ by Lemma~\ref{lem:BorelFactorization}.
\end{proof}

\section{The toric ring of a principal Borel ideal}\label{sec:Principal}

De Negri shows in~\cite{DN99} that Sturmfels' concept of \textit{sorting} from~\cite{S96} yields a Gr\"obner basis of quadrics for the toric ideal of a principal Borel ideal.  The monomial order yielding this Gr\"obner basis of quadrics is called the \textit{sorting order}, and may be different from both graded reverse lexicographic order and lexicographic order on $\K[\bT]$.  The sorting order depends, a priori, on the ideal one starts with.  This makes it difficult to extend sorting to the multi-Rees algebra, where we have to consider multiple ideals simultaneously.  See~\cite{Sosa14} where this approach is taken, yielding some partial results for multi-Rees algebras of principal Borel ideals.

In this section we show that a modification of Sturmfels' sorting procedure from~\cite{S96} yields a Gr\"obner basis of quadrics under \textit{lexicographic} order.  Using this we show that the Rees algebra of a principal Borel ideal has a quadratic squarefree initial ideal under lexicographic order.  In the following sections we apply this to multi-Rees algebras.

Let $I$ be an ideal of the polynomial ring $\K[\x]$, and fix the \textit{graded reverse lexicographic} (grevlex) order on $\K[\x]$.
Let $G$ be a minimal set of generators for $I$; then the toric ring of $I$ is $\K[I]=\K[m:m\in G]$.  As in Section~\ref{sec:ToricMaps}, we consider the map
\[
\K[T_m:m\in G]=\K[\bT]\xrightarrow{\phi} \K[I]=\K[m:m\in G],
\]
where $\K[\bT]$ is the polynomial ring with an indeterminate corresponding to every generator of $I$, $\K[I]$ is the sub-algebra of the polynomial ring $\K[\x]$ generated by the elements of $G$, and the map is defined on indeterminates by $\phi(T_m)=m$.  The grevlex order on $\K[\x]$ induces an ordering on the variables $\bT$ by $T_n\prec T_m$ if and only if $n\prec_{\grevlex} m$ in $\K[\x]$.  On top of this ordering of the variables of $\K[\bT]$ we will put the \textit{lexicographic} (lex) ordering on $\K[\bT]$.  For two monomials $\bT^{\bgamma},\bT^{\bgamma'}$ in $\K[\bT]$ we write $\bT^{\gamma}\prec_{\lex} \bT^{\gamma'}$ if $\bT^{\gamma}$ is smaller than $\bT^{\gamma'}$ in lex order.

We consider the toric ideal $J=J_{\K[I]}$.  Let $\mu\in\K[I]$.  The graded component $J_\mu$ is spanned by differences of the form $\prod_{i=1}^k T_{m_i}-\prod_{i=1}^k T_{n_i}$, where $\prod_{i=1}^k m_i=\prod_{i=1}^k n_i=\mu$ and $m_1,\ldots,m_k,n_1,\ldots,n_k\in \gens(I)=\Borel(M)$ for some monomial $M$.  By Lemma~\ref{lem:BorelFactorization}, $\mu$ is the product of $k$ monomials from $\Borel(M)$ if and only if $\mu\in\Borel(M^k)$.

\begin{definition}\label{def:BorelSort}
	Suppose $M$ is a monomial and $\mu\in\Borel(M^k)$.  Then the \blgsort{} of $\mu$ is the factorization of $\mu$ into $k$ factors from $\Borel(M)$ which is produced by the recursive Algorithm~\ref{alg:BS}.  We write this factorization as the \textit{ordered} list
	\[
	\bs(M,\mu):=\{\mu_1,\ldots,\mu_k\},
	\]
	where $\mu=\mu_1\cdots \mu_k$.
\end{definition}

Given a monomial $\mu\in\Borel(M^k)$, we will see in Theorem~\ref{thm:LexOrderReesAlgebraPrincipleBorel} that the factorization produced by $\bs(M,\mu)$ is the smallest factorization of $\mu$ under lexicographic order whose factors are in $\Borel(M)$.  We first give an example and verify some basic properties of $\bs(M,\mu)$ which follow from Algorithm~\ref{alg:BS}.

\begin{remark}
As in~\cite{S96} and~\cite{DN99}, an effective way to visualize the factorization $\mu=\prod_{i=1}^k \mu_i$ is to consider $\mu_1,\ldots,\mu_k$ as the $k$ rows of a $k\times d$ tableaux, where $d$ is the degree of $\mu$.  Each entry of the tableaux is filled with a variable of the underlying polynomial ring, and the product of the variables in each row is $\mu_i$.  We illustrate this in Example~\ref{ex:BS}.
\end{remark}

\begin{remark}
An implementation of Algorithm~\ref{alg:BS} in the computer algebra system Macaulay2~\cite{M2} can be found under the Research tab on the first author's website, along with a script to verify Example~\ref{ex:BS}.
\end{remark}

\begin{algorithm}
	\KwIn{A monomial $M$ of degree $d$ and a monomial $\mu\in \Borel(M^k)$}
	\KwOut{A list $\{\mu_1,\ldots,\mu_k\}$ of monomials, the \blgsort{} of $\mu$}
	\If{$\mu=x_i^{dk}$ for some $1\le i\le n$}{
		\For{$1\le i\le k$}{$\mu_i\leftarrow x_i^d$}
	}
	\Else{
		$x_s\leftarrow$ variable with largest index dividing $\mu$\;
		$A\leftarrow$ exponent of $x_s$ in $\mu$\;
		$q\leftarrow$ quotient when $A$ is divided by $k$\;
		$r\leftarrow$ remainder when $A$ is divided by $k$\;
		\If{$r>0$}{
			$M_\up\leftarrow$ $\dfrac{1}{x_s^q}\cdot ($least monomial in $\Borel(M)$ under Borel order not divisible by $x_i$ for any $i>s$ and whose exponent on $x_s$ is $q)$\;
			$\mu_\up\leftarrow$ least monomial in $\Borel(M_\up^{k-r})$ under Borel order dividing $\mu$\;
			$M_\lw\leftarrow$ $\dfrac{1}{x_s^{q+1}}\cdot ($least monomial in $\Borel(M)$ under Borel order not divisible by $x_i$ for any $i>s$ and whose exponent on $x_s$ is $q+1)$\;
			$\mu_\lw\leftarrow$ $\mu/(\mu_\up\cdot x_s^A)$\;
			$\{U_1,\ldots,U_{k-r}\}\leftarrow\bs(M_\up,\mu_\up)$\;
			$\{D_1,\ldots,D_r\}\leftarrow \bs(M_\lw,\mu_\lw)$\;
			\For{$1\le i\le k-r$}{$\mu_i\leftarrow x_s^q\cdot U_i$}
			\For{$k-r< i\le k$}{$\mu_i\leftarrow x_s^{q+1}\cdot D_{i-k+r}$}
		}
		\Else{
			$M_\lft\leftarrow$ $\dfrac{1}{x_s^q}\cdot ($least monomial in $\Borel(M)$ under Borel order not divisible by $x_i$ for any $i>s$ and whose exponent on $x_s$ is $q)$\;
			$\mu_\lft\leftarrow\mu/x_s^A$\;
			$\{L_1,\ldots,L_k\}\leftarrow\bs(M_\lft,\mu_\lft)$\;
			\For{$1\le i\le k$}{
				$\mu_i\leftarrow x_s^q\cdot L_i$
			}
		}
	}
	\KwRet{\{$\mu_1,\ldots,\mu_k$\}}
	\caption{$\bs$\label{alg:BS}}
\end{algorithm}

\begin{example}[Illustration of Algorithm~\ref{alg:BS}]\label{ex:BS}
	Let $S=\K[x_0,x_1,x_2,x_3,x_4]$, $M=x_1x_3^2x_4^2$, and $\mu=x_0^2x_1^5x_2^{13}x_3^7x_4^3\in\Borel(M^6)$.  A computation in Macaulay2~\cite{M2} (see the script under the research tab on the first author's website) indicates that there are $4,742$ factorizations of $\mu$ into a product of six factors, each of which belongs to $\Borel(M)$.  As we will see in Theorem~\ref{thm:LexOrderReesAlgebraPrincipleBorel}, Algorithm~\ref{alg:BS} picks out the factorization which is minimal with respect to the lexicographic order on these factorizations (considered as ordered tuples).
	
	To visualize the factorization $\bs(M,\mu)=\{\mu_1,\ldots,\mu_6\}$ produced by Algorithm~\ref{alg:BS}, we fix a $6\times 5$ tableaux which we will fill with the variables of $\mu$ (see the top diagram in Figure~\ref{fig:BS}).  In the first level of recursion in Algorithm~\ref{alg:BS}, the variable with largest index dividing $\mu$ is $x_4$, which appears with exponent $3$ in $\mu$, hence $A=3$, $q=0$, and $r=3$.  At the end of the $r>0$ branch, which we will follow next, we see that $\mu_1,\mu_2,$ and $\mu_3$ do not get an $x_4$ while $\mu_3,\mu_4,$ and $\mu_5$ each receive a single $x_4$.  So we fill in the last entry of each of these rows with an $x_4$.  We split the remaining portion of the $6\times 5$ tableaux into two blocks - a $3\times 5$ `upper' tableaux and a $3\times 4$ `lower' tableaux (outlined in red and blue, respectively, in the top diagram of Figure~\ref{fig:BS}).
	
	We now follow the $r>0$ branch.  The least monomial under Borel order in $\Borel(M)$ not divisible by $x_4$ is $x_1x_3^4$, so $M_\up=x_1x_3^4$.  Now the least monomial in $\Borel(M_\up^{k-r})=\Borel(x_1^3x_3^{12})$ which divides $\mu$ is $\mu_\up=x_1^3x_2^5x_3^7$.  This latter monomial is the one which we will use to fill the `upper' $3\times 5$ tableaux (outlined in red in the top diagram of Figure~\ref{fig:BS}).  Thus we return to the beginning of the algorithm with the monomials $M_\up=x_1x_3^4$ and $\mu_\up=x_1^3x_2^5x_3^7$ in order to fill this upper tableaux.
	
	The remaining monomial $\mu_\lw=\mu/(\mu_\up x_4^3)=x_0^2x_1^2x_2^8$ is what we will use to fill the `lower' $3\times 4$ tableaux (outlined in blue in the top diagram of Figure~\ref{fig:BS}).  We again return to the beginning of the algorithm with the monomials $M_\lw=x_1x_3^2$ and $\mu_\lw=x_0^2x_1^2x_2^8$ in order to fill this lower tableaux.
	
	The recursion to fill the entire $6\times 5$ tableaux has three levels; the monomials $M_\up,\mu_\up,M_\lw,\mu_\lw$ or $M_\lft,\mu_\lft$ are shown for each level in a tree structure in the bottom diagram of Figure~\ref{fig:BS} (this should be read right to left).  The corresponding subdivisions of the tableaux are shown in the top diagram of Figure~\ref{fig:BS}.  The red, blue, and green outlines indicate (respectively) $\up,\lw,$ and $\lft$ recursive calls to Algorithm~\ref{alg:BS}.
	
	At the completion of Algorithm~\ref{alg:BS}, we can read off the monomials $\mu_1,\ldots,\mu_6$ by taking the product of the variables in the corresponding rows of the top diagram in Figure~\ref{fig:BS}.  We end up with $\mu_1=\mu_2=x_1x_2^2x_3^2,\mu_3=x_1x_2x_3^3,\mu_4=x_1^2x_2^2x_4,$ and $\mu_5=\mu_6=x_0x_2^3x_4$.  Thus
	\begin{multline*}
	\bs(x_1x_3^2x_4^2,x_0^2x_1^5x_2^{13}x_3^7x_4^3)= \\ \{x_1x_2^2x_3^2,x_1x_2^2x_3^2,x_1x_2x_3^3,x_1^2x_2^2x_4,x_0x_2^3x_4,x_0x_2^3x_4\}.
	\end{multline*}
	\begin{figure}
		\begin{tikzpicture}
		\foreach \x in {1,2,...,5}
		\foreach \y in {1,...,6}
		{
			\draw (\x,\y) +(-.5,-.5) rectangle ++(.5,.5);
		}
		\foreach \y in {1,2,3}
		{
			\node at (5,\y) {$x_4$};
		}
		\foreach \x in {4,5}
		\foreach \y in {4,5,6}
		{
			\node at (\x,\y) {$x_3$};
		}
		\node at (3,4){$x_3$};
		\foreach \x in {2,3}
		\foreach \y in {5,6}
		{
			\node at (\x,\y){$x_2$};
		}
		\foreach \x in {3,4}
		\foreach \y in {1,2,3}
		{
			\node at (\x,\y){$x_2$};
		}
		\foreach \y in {3,...,6}
		{
			\node at (1,\y) {$x_1$};
		}
		\node at (2,4){$x_2$};
		\node at (2,3){$x_1$};
		\foreach \y in {1,2}
		{
			\node at (2,\y){$x_2$};
		}
		\foreach \y in {1,2}
		{
			\node at (1,\y){$x_0$};
		}
		\foreach \y in {1,...,6}
		{
			\node at (0,7-\y) {$\mu_\y$};
		}
		\def\eps{.025}
		\def\fs{1.5}
		\def\fb{2.5}
		\draw[very thick,red] (.5,3.5+\eps)--(5.5,3.5+\eps)--(5.5,6.5)--(.5,6.5)--cycle;
		\draw[very thick,blue] (.5,.5)--(4.5,.5)--(4.5,3.5-\eps)--(.5,3.5-\eps)--cycle;
		\draw[very thick,red] (.5+\fs*\eps,4.5+\eps)--(3.5,4.5+\eps)--(3.5,6.5-\fs*\eps)--(.5+\fs*\eps,6.5-\fs*\eps)--cycle;
		\draw[very thick,blue] (.5+\fs*\eps,3.5+\fb*\eps)--(2.5,3.5+\fb*\eps)--(2.5,4.5-\eps)--(.5+\fs*\eps,4.5-\eps)--cycle;
		\draw[very thick, green] (.5+\fb*\eps,4.5+\fb*\eps)--(1.5,4.5+\fb*\eps)--(1.5,6.5-\fb*\eps)--(.5+\fb*\eps,6.5-\fb*\eps)--cycle;
		\draw[very thick, green] (.5+\fb*\eps,3.5+1.6*\fb*\eps)--(1.5,3.5+1.6*\fb*\eps)--(1.5,4.5-\fb*\eps)--(.5+\fb*\eps,4.5-\fb*\eps)--cycle;
		\draw[very thick,red] (.5+\fs*\eps,2.5+\eps)--(2.5,2.5+\eps)--(2.5,3.5-\fb*\eps)--(.5+\fs*\eps,3.5-\fb*\eps)--cycle;
		\draw[very thick, blue] (.5+\fs*\eps,.5+\fs*\eps)--(1.5,.5+\fs*\eps)--(1.5,2.5-\eps)--(.5+\fs*\eps,2.5-\eps)--cycle;
		\end{tikzpicture}
		
		\vspace{20 pt}
		
		\begin{tikzpicture}
		\tikzstyle{level 1}=[sibling distance=100 pt]
		\tikzstyle{level 2}=[sibling distance=50 pt]
		\node[rectangle, draw] {\parbox{75 pt}{$M=x_1x_3^2x_4^2$ \\[3 pt] $\mu=x_0^2x_1^5x_2^{13}x_3^7x_4^3$}}
		[level distance=70 pt]
		[grow=left]
		child {node[rectangle,draw=red,very thick] {\parbox{50 pt}{$M=x_1x_3^4$ \\[3 pt] $\mu=x_1^3x_2^5x_3^{7}$}}
			child {node[rectangle,draw=red,very thick] {\parbox{45 pt}{$M=x_1x_2^2$ \\[3 pt] $\mu=x_1^2x_2^4$}}
				child {node[rectangle,draw=green,very thick] {\parbox{35 pt}{$M=x_1$\\[3 pt] $\mu=x_1^2$}}}
			}
			child {node[rectangle,draw=blue,very thick] {\parbox{45 pt}{$M=x_1x_2$ \\[3 pt] $\mu=x_1x_2$}}
				child {node[rectangle,draw=green,very thick] {\parbox{35 pt}{$M=x_1$\\[3 pt] $\mu=x_1$}}}
			}
		}
		child {node[rectangle,draw=blue,very thick] {\parbox{50 pt}{$M=x_1x_3^3$ \\[3 pt] $\mu=x_0^2x_1^2x_2^{8}$}}
			child {node[rectangle,draw=red,very thick] {\parbox{35 pt}{$M=x_1^2$\\[3 pt] $\mu=x_1^2$}}}
			child {node[rectangle,draw=blue,very thick] {\parbox{35 pt}{$M=x_1$\\[3 pt] $\mu=x_0^2$}}}
		};
		\end{tikzpicture}
		\caption{Illustration of Algorithm~\ref{alg:BS} in Example~\ref{ex:BS}; this should be read right to left.  The red, blue, and green outlines indicate (respectively) $\up,\lw,$ and $\lft$ recursive calls to Algorithm~\ref{alg:BS}.}\label{fig:BS}
	\end{figure}
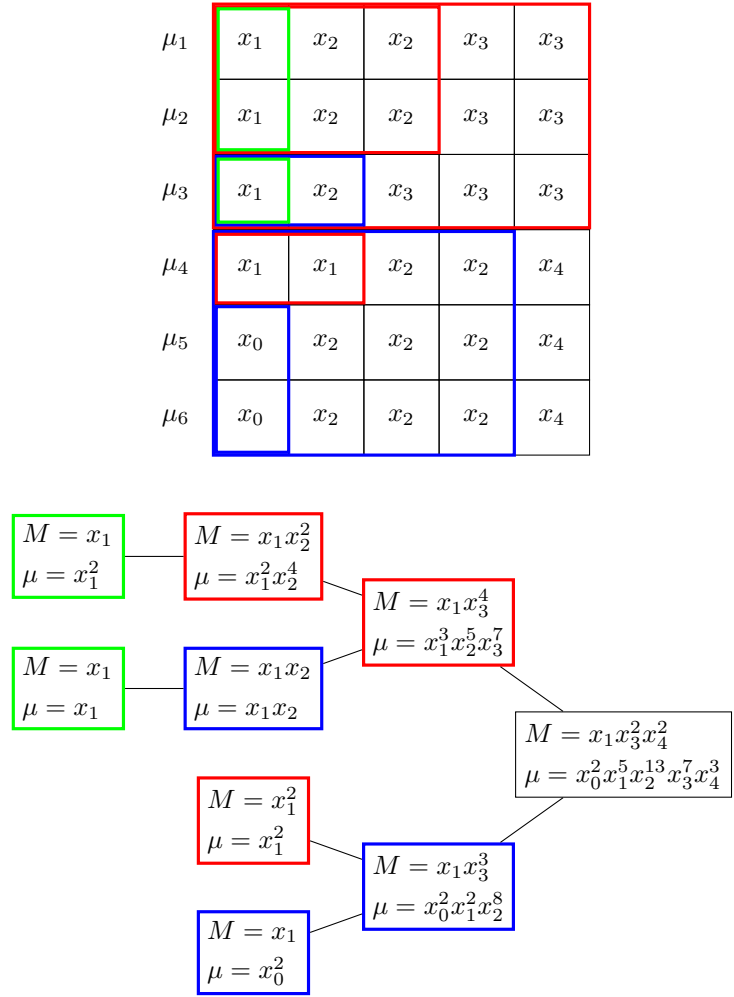

\end{example}

The following lemma shows that $M_\up,M_\lw,$ and $M_\lft$ are well-defined in Algorithm~\ref{alg:BS}.  That $\mu_\up$ is well-defined follows from Lemma~\ref{lem:BorelMultidegreeDivision}.

\begin{lemma}\label{lem:SelectingUpperAndLowerMonomials}
	Let $M,\mu\in \K[\x]$, and suppose $\mu\in \Borel(M^k)$.  Suppose that $x_s$ is the variable of largest index dividing $\mu$, $x_s$ appears in $\mu$ with exponent $A$, and $s>1$.  Put $A=qk+r$, where $q\in\mathbb{Z}_{\ge 0}$ and $0\le r<k$.  There is a unique minimal monomial $\gamma$ in $\Borel(M)$, under Borel order, satisfying
	\begin{enumerate}
		\item $x_i$ does not divide $\gamma$ for any $i>s$
		\item $x_s$ appears with exponent $q$ in $\gamma$ (if $r=0$) or with exponent $q+1$ in $\gamma$ (if $r>0$)
	\end{enumerate}
\end{lemma}
\begin{proof}
	Since $\mu\in \Borel(M^k)$, it follows that $qk+r=A=\sigma_s(\mu)\le \sigma_s(M^k)=k\sigma_s(M)$.  Hence if $r=0$, $q\le\sigma_s(M)$, and if $r>0$, $q+1\le \sigma_s(M)$.  Now define $\gamma$ by
	\begin{enumerate}
		\item $\sigma_j(\gamma)=0$ for $j>s$
		\item $\sigma_s(\gamma)=q$ (if $r=0$) or $\sigma_s(\gamma)=q+1$ (if $r>0$)
		\item $\sigma_i(\gamma)=\sigma_i(M)$ for $i<s$.
	\end{enumerate}
	Explicitly, $\gamma$ is not divisible by $x_i$ for $i>s$, the exponent of $x_s$ in $\gamma$ is $q$ (if $r=0$) or $q+1$ (if $r>0$), the exponent of $x_{s-1}$ in $\gamma$ is $\sigma_{s-1}(M)-q$ (if $r=0$) or $\sigma_{s-1}(M)-q-1$ (if $r>0$), and the exponent of $x_i$ ($i<s-2$) in $\gamma$ is the same as the exponent of $x_i$ in $M$.  The condition $s>1$ is necessary -- if $s=1$ we will not reach (3) and $\gamma$ will not have the correct degree.  By Lemma~\ref{lem:BorelMembership}, $\gamma\in\Borel(M)$.  Clearly $\gamma$ is minimal under Borel order with respect to the desired properties.
\end{proof}


\begin{lemma}\label{lem:BorelSortProperties}
	Let $M$ be a monomial of degree $d$ and $\mu\in\Borel(M^k)$.  Then Algorithm~\ref{alg:BS} terminates on this input, returning $\bs(M,\mu)=\{\mu_1,\ldots,\mu_k\}$ with the properties:
	\begin{enumerate}
		\item\label{p:1} $\mu_1\cdots \mu_k=\mu$
		\item\label{p:2} $\mu_1\succeq_\grevlex\cdots\succeq_\grevlex \mu_k$
		\item\label{p:3} $\mu_i\in\Borel(M)$ for $i=1,\ldots,k$
	\end{enumerate}
\end{lemma}
\begin{proof}
	Algorithm~\ref{alg:BS} terminates because the number of variables in $\mu$ is reduced by one during each step of the recursion, which will eventually trigger the leading `if' statement.
	
	We prove \eqref{p:1},~\eqref{p:2} and~\eqref{p:3} by induction on the number of variables dividing $\mu$.  If $\mu$ is only divisible by a single variable, so $\mu=x_i^{dk}$ for some $1\le i\le n$, then $\mu_1=\mu_2=\cdots=\mu_k=x_i^d$.  Statements~\eqref{p:1} and~\eqref{p:2} are clear, so we need only show that $x_i^d\in\Borel(M)$.  By Lemma~\ref{lem:BorelMembership}, $\sigma_i(\mu)=dk\le \sigma_i(M^k)$.  Since this is as large as $\sigma_i(\mu)$ can be for a monomial of degree $dk$, $\sigma_i(M^k)=dk$.  It follows that $\sigma_i(x_i^d)=\sigma_i(M)=d$.  Moreover, $\sigma_j(x_i^d)=d=\sigma_j(M)$ for any $j<i$ and $\sigma_j(x_i^d)=0\le \sigma_j(M)$ for any $j>i$.  It follows from Lemma~\ref{lem:BorelMembership} that $x_i^d\in \Borel(M)$.
	
	Now suppose that $k>1$ and $\mu$ is divisible by more than one variable.  Let $s$ be the largest index of a variable dividing $\mu$, let $A$ be the exponent of $x_s$ in $\mu$, and put $A=qk+r$ (as in Algorithm~\ref{alg:BS}).  First we show~\eqref{p:1}.  If $r>0$, then by induction on the number of variables, $U_1\cdots U_{k-r}=\mu_\up$ and $D_1\cdots D_r=\mu_\lw$.  Since $\mu_i=x_s^q\cdot U_i$ for $1\le i\le k-r$ and $\mu_i=x_s^{q+1}\cdot D_{i-k+r}$ for $k-r< i\le k$,
	\[
	\mu_1\ldots \mu_k=x_s^{qk+r}\mu_\up\mu_\lw=\mu_\lw(x_s^A\cdot \mu_\up)=\mu.
	\]
	If $r=0$ then by induction $L_1\cdots L_k=\mu_\lft$.  So $\mu_1\cdots \mu_k=x_s^{qk}\mu_\lft=\mu$.
	
	Now we show~\eqref{p:2}.  If $r>0$, then by induction $U_1\succeq_\grevlex \cdots \succeq_\grevlex U_{k-r}$ and $D_1\succeq_\grevlex\cdots\succeq_\grevlex D_r$.  Thus $\mu_1\succeq_\grevlex\cdots\succeq_\grevlex \mu_{k-r}$ and $\mu_{k-r+1}\succeq_\grevlex\cdots \succeq_\grevlex \mu_{k}$.  Moreover, the least variable dividing $\mu_{k-r}$, namely $x_s$, has exponent $q$ in $\mu_{k-r}$ and exponent $q+1$ in $\mu_{k-r+1}$; hence $\mu_{k-r}\succeq_\grevlex \mu_{k-r+1}$.  Now suppose $r=0$.  By induction $L_1\succeq_{\grevlex}\cdots \succeq_\grevlex L_k$.  Hence $\mu_1\succeq_{\grevlex}\cdots\succeq_\grevlex \mu_k$.
	
	Finally we show~\eqref{p:3}.  If $r>0$, then by induction $U_i\in\Borel(M_\up)$ for $i=1,\ldots,k-r$ and $D_i\in \Borel(M_\lw)$ for $i=1,\ldots,r$.  Since $x_s^qM_\up\in \Borel(M)$ (by its definition in Algorithm~\ref{alg:BS}) it follows that $\mu_i=x_s^q\cdot U_i\in\Borel(M)$ for $i=1,\ldots,k-r$.  Likewise, $x_s^{q+1}M_\lw\in\Borel(M)$ (again by its definition in Algorithm~\ref{alg:BS}) so it follows that $\mu_i=x_s^{q+1}D_{i-k+r}\in\Borel(M)$ for $i=k-r+1,\ldots,k$.  If $r=0$ then $L_1,\ldots,L_k\in\Borel(M_\lw)$ by induction.  Since $x_s^qM_\lw\in\Borel(M)$ (by its definition in Algorithm~\ref{alg:BS}), it follows that $\mu_i=x_s^qL_i\in\Borel(M)$ for $i=1,\ldots,k$.
\end{proof}

Now we return to the toric ring.  Let $M\in \K[\x]$ be a fixed monomial, $I=\langle\Borel(M)\rangle\subset \K[\x]$, $\mu\in \K[\x]$, $\bT=\{T_m:m\in\Borel(M)\}$, and $\phi:\K[\bT]\to \K[\x]$ be defined by $\phi(T_m)=m$.  Order the variables of $\K[\bT]$ by $T_{m}\succ T_{m'}$ if $m\succ_{\grevlex} m'$.  Let $\prec$ be the lexicographic monomial order on $\K[\bT]$ with respect to this ordering of the variables.  Set 
\begin{equation}\label{eq:ReesAlgebraGB}
\calB(M)=\{T_m T_n-T_{\frac{x_i}{x_j}m}T_{\frac{x_j}{x_i}n}:x_j\mid m,x_i\mid n,\mbox{ and }\frac{x_i}{x_j}m,n\in\Borel(M)\}.
\end{equation}
Let $\vec{\Gamma}_{\mu,\calB}$ be the directed graph associated to $\mu$ and $\calB$ (see Definition~\ref{def:FiberGraph}).  Notice that by Lemma~\ref{lem:BorelFactorization}, $\Gamma_{\mu,\calB}$ is empty unless $\mu\in\Borel(M^k)$ for some $k\ge 1$.

\begin{definition}\label{def:LexSink}
	Fix monomials $M\in \K[\x],\mu\in\Borel(M^k)$, and the toric map $\phi:\K[T_m:m\in\Borel(M)]\to \K[\x]$ defined by $\phi(T_m)=m$.  Let $\bs(M,\mu)=\{\mu_1,\ldots,\mu_k\}$.  We define $\bT^\mu_{\min}=\prod_{i=1}^k T_{\mu_i}$.
\end{definition}

\begin{theorem}\label{thm:LexOrderReesAlgebraPrincipleBorel}
	Fix monomials $M\in \K[\x],\mu\in\Borel(M^k)$, and the toric map $\phi:\K[T_m:m\in\Borel(M)]\to \K[\x]$ defined by $\phi(T_m)=m$.  We put the monomial order $\prec_\grevlex$ on $\K[\x]$.  This induces the order $T_m\succ T_n$ if $m\succ_\grevlex n$ on the variables of $\K[\bT]=\K[T_m:m\in\Borel(M)]$.  On top of this ordering of the variables we put the lexicographic monomial order $\prec_\lex$ on $\K[\bT]$.  Then $\bT^\mu_{\min}$ is the unique sink in $\vec{\Gamma}_{\mu,\calB}$, where $\calB=\calB(M)$ is the set in~\eqref{eq:ReesAlgebraGB}.  In particular, $\calB$ is a Gr\"obner basis for $J_{\K[I]}$ with respect to $\prec_\lex$, where $I=\langle \Borel(M)\rangle$.  Additionally, the set of quadrics
	\[
	\{T_mT_n-T_{\mu_1}T_{\mu_2}:m,n\in\Borel(M),\bs(M,mn)=\{\mu_1,\mu_2\}\}
	\]
	is also a Gr\"obner basis for $J_{\K[I]}$ with respect to $\prec_\lex$.
\end{theorem}

\begin{discussion}\label{disc:1}
	In the proof of Theorem~\ref{thm:LexOrderReesAlgebraPrincipleBorel} we will use the notation of Algorithm~\ref{alg:BS} to break apart the various inductive steps.  In order to streamline the proof, we discuss here how the fiber graphs of the different monomials in Algorithm~\ref{alg:BS} interact with each other.
	
	For fixed monomials $M$ and $\mu\in\Borel(M^k)$, we let $x_s$,$A$,$q$, and $r$ be as in Algorithm~\ref{alg:BS}.  If $r>0$ we have the monomials $M_\up,\mu_\up,M_\lw,$ and $\mu_\lw$ and if $r=0$ we have the monomials $M_\lft$ and $\mu_\lft$ as in Algorithm~\ref{alg:BS}.
	
	If $r=0$ we write $\bT_\lft$ for $\{T_\ell:\ell\in\Borel(M_\lft)$ and we define the injective map $\phi_\lft:\K[\bT_\lft]_{\mu_\lft}\to \K[\bT]_\mu$ by
	\[
	\phi_\lft(\prod_{i=1}^k T_{\ell_i})=\prod_{i=1}^k T_{\ell_ix_s^q},
	\]
	extended linearly.  Under this map, it follows from Definition~\ref{def:LexSink} and Algorithm~\ref{alg:BS} that $\phi_\lft(\bT^{\mu_\lft}_{\min})=\bT^\mu_{\min}$.
	
	Let $\calB_\lft=\calB(\mu_\lft)$ be defined as in~\eqref{eq:ReesAlgebraGB}.  Associated to the map $\phi_\lft$ we can identify $\vec{\Gamma}_{\mu_\lft,\calB_\lft}$ with a directed subgraph of $\vec{\Gamma}_{\mu,\calB}$ which we will denote by $\vec{\Gamma}_\lft$.  To obtain $\vec{\Gamma}_\lft$ simply apply the map $\phi_\lft$ to the vertices of $\vec{\Gamma}_{\mu\lft,\calB_\lft}$ and connect two vertices with an edge in $\vec{\Gamma}_\lft$ if their preimages are connected by an edge.  We claim $\vec{\Gamma}_\lft$ is a subgraph of $\vec{\Gamma}_{\mu,\calB}$; to see this, notice that if
	\[
	B=T_\ell T_{\ell'}-T_{\frac{x_i}{x_j}\ell}T_{\frac{x_j}{x_i}\ell'}\in\calB_\lft,
	\]
	then $\phi_\lft(B)\in\calB$.  Thus a directed edge between vertices of $\vec{\Gamma}_{\mu_\lft,\calB_\lft}$ gets sent by $\phi_\lft$ to a directed edge between vertices of $\vec{\Gamma}_{\mu,\calB}$.
	
	We now define analagous maps if $r>0$.  We write $\bT_\up$ and $\bT_\lw$ for the sets of variables $\{T_m:m\in\Borel(M_\up)\}$ and $\{T_n:n\in\Borel(M_\lw)\},$ respectively.  We define an injective map $\phi_{\up,\lw}:\K[\bT_\up]_{\mu_\up}\times \K[\bT_\lw]_{\mu_\lw}\to \K[\bT]_\mu$ by 
	\[
	\phi_{\up,\lw}(\prod_{i=1}^{k-r} T_{m_i},\prod_{i=k-r+1}^k T_{n_i})=\prod_{i=1}^{k-r} T_{m_ix_s^{q}}\prod_{i=k-r+1}^{k} T_{n_ix_s^{q+1}},
	\]
	extended linearly.  Under this map, it follows from Definition~\ref{def:LexSink} and Algorithm~\ref{alg:BS} that $\phi_{\up,\lw}(\bT^{\mu_\up}_{\min},\bT^{\mu_\lw}_{\min})=\bT^\mu_{\min}$.
	
	Let $\calB_\up=\calB(M_\up)$ and $\calB_\lw=\calB(M_\lw)$ be as in~\eqref{eq:ReesAlgebraGB}.  Just as in the case $r=0$, we can use the map $\phi_{\up,\lw}$ to identify the Cartesian product $\vec{\Gamma}_{\mu_\up,\calB_\up}\times \vec{\Gamma}_{\mu_\lw,\calB_\lw}$ with a subgraph of $\Gamma_{\mu,\calB}$, which we denote as $\vec{\Gamma}_{\up,\lw}$.
\end{discussion}

\begin{proof}[Proof of Theorem~\ref{thm:LexOrderReesAlgebraPrincipleBorel}]
	We will show by induction on the number of variables dividing $\mu$ that $\bT^\mu_{\min}$ is the unique sink in $\vec{\Gamma}_{\mu,\calB}$.  First suppose that $\mu$ is only divisible by a single variable.  Then $\bT^\mu_{\min}$ is the unique monomial in $\K[\bT]$ so that $\phi(\bT^\mu_{\min})=\mu$, hence it is of course the unique sink.
	
	Now suppose that $\mu$ is divisible by more than one variable and let $x_s$ be the variable of largest index dividing $\mu$, $A$ the exponent of $x_s$ in $\mu$, and set $A=qk+r$ as in Algorithm~\ref{alg:BS}.  Further suppose that $N=\{N_1,\ldots,N_k\}$, $N_1\succeq_\grevlex\cdots \succeq_\grevlex N_k$, and $\bT_N=\prod_{i=1}^k T_{N_i}$ satisfies $\phi(\bT_N)=\mu$ and $\bT_N$ is a sink.  We will show that $\bT_N=\bT^\mu_{\min}$.
	
	First suppose that there exist indices $i<j$ so that $x_s$ appears in $N_i$ with exponent $n_i$, $x_s$ appears in $N_j$ with exponent $n_j$, and $n_j-n_i\ge 2$.  Now $N_j$ divides $\mu$, $\sigma_s(N_j)>\sigma_s(N_i)$, and hence $\sigma_s(M')>\sigma_s(N_i)$, where $M'$ is the unique minimal monomial in $\Borel(M)$ under Borel order which divides $\mu$.  Clearly $s$ is the largest possible index where we can have $\sigma_s(N_i)<\sigma_s(M')$, as $x_s$ is the variable of largest index dividing $\mu$.  It follows from Lemma~\ref{lem:BorelMultidegreeDivision} that $N_i$ must be divisible by some variable $x_h$ where $h<s$ and $\frac{x_s}{x_h}N_i$ divides $\mu$ and is in $\Borel(M)$.  Let
	\[
	\bT_{N'}=T_{\frac{x_s}{x_h}N_i}T_{\frac{x_h}{x_s}N_j}\prod_{u\neq i,j}T_{N_u}.
	\]
	Clearly $\frac{x_s}{x_h}N_i\preceq_\grevlex N_i$ and $\frac{x_h}{x_s}N_j\preceq_\grevlex N_i$ since the exponent of $x_s$ in $N_i$ is smaller in both cases.  It follows that
	\[
	\bT_{N'}\prec \bT_{N},
	\]
	hence the directed edge $\bT_N\to\bT_{N'}$ appears in $\vec{\Gamma}_{\mu,\calB}$ and $\bT_N$ is not a sink.  Thus if $\bT_N$ is a sink, we may assume that the exponents of $x_s$ in $N_i$ for $i=1,\ldots,k$ differ by at most one.  We show that the exponents must be either $q$ or $q+1$.  Suppose there exists an index $i$ so that the exponent of $x_s$ in $N_i$ is $n_i$ and $n_i<q$.  Then, by the pigeonhole principle, there is an index $j$ so that the exponent of $x_s$ in $N_j$, call it $n_j$, satisfies $n_j> q$.  Since $n_j-n_i\ge 2$, $\bT_N$ is not a sink.  Similarly, suppose there is an index $j$ so that $n_j>q+1$.  Then there is again an index $i$ so that $n_i<q+1$ and $\bT_N$ is not a sink.  So the exponent of $x_s$ must be either $q$ or $q+1$ in each $N_i$.
	
	If $r=0$, we may thus assume that the exponent of $x_s$ in $N_i$ is $q$ for $1\le i\le k$.  As in Discussion~\ref{disc:1}, let $\bT_\lft$ be the set of variables $\{T_m:m\in\Borel(M_\lft)\}$.  By induction on the number of variables, it follows that $\bT^{\mu_\lft}_{\min}\in K[\bT_\lft]_{\mu_\lft}$ is the unique sink of $\vec{\Gamma}_{\mu_\lft,\calB_{\lft}}$.  As described in Discussion~\ref{disc:1}, we may use the map
	\[
	\phi_\lft:K[\bT_\lft]_{\mu_\lft}\to K[\bT]_\mu
	\]
	to identify $\vec{\Gamma}_{\mu_\lft,\calB_{\lft}}$ with a directed subgraph of $\vec{\Gamma}_{\mu,\calB}$ which we denote by $\vec{\Gamma}_\lft$.  Since $T_N$ is a sink of $\vec{\Gamma}_{\mu,\calB}$ by assumption, it must also be a sink of $\vec{\Gamma}_\lft$.  However, $\vec{\Gamma}_\lft$ has a unique sink by induction, namely $\phi_\lft(\bT^{\mu_\lft}_{\min})$, which we observed in Discussion~\ref{disc:1} is the same as $\bT^\mu_{\min}$.  Thus $\bT_N=\phi_\lft(\bT^{\mu_\lft}_{\min})=\bT^\mu_{\min}$, as desired.
	
	If $r>0$, we may assume that the exponent of $x_s$ in $N_i$ is $q$ for $i=1,\ldots,k-r$ and $q+1$ for $i=k-r+1,\ldots,k$.  Let $N'_i=\frac{N_i}{x_s^q}$ for $i=1,\ldots,k-r$ and $N'_i=\frac{N_i}{x_s^{q+1}}$ for $i=k-r+1,\ldots,k$.  Put 
	\[
	\mu^N_\up=\prod_{i=1}^{k-r} N'_i,\qquad \mu^N_\lw=\prod_{i=k-r+1}^k N'_i,
	\]
	and let $M_\up,\mu_\up,M_\lw,\mu_\lw$ be as in Algorithm~\ref{alg:BS}.  Notice that $\mu^N_\up$ and $\mu^N_\lw$ are only divisible by variables whose index is \textit{strictly less} than $s$, so $N'_i\in\Borel(M_\up)$ for $i=1,\ldots,k-r$ and $\mu^N_\up\in \Borel(M_\up^{k-r})$.  Recall that $\mu_\up$ is, by definition, the least monomial under Borel order in $\Borel(M_\up^{k-r})$ which divides $\mu$, so $\mu^N_\up\in\Borel(\mu_\up)$ by Lemma~\ref{lem:BorelMultidegreeDivision}.
	If $\mu^N_\up\neq \mu_\up$, 
	 then by Lemma~\ref{lem:BorelFactorizationMultidegreeDivision}, there are indices $u<v$ and an index $1\le i\le k-r$ so that
	 \begin{itemize}
	 	\item $\frac{x_v}{x_u}N'_i\in\Borel(M_\up)$,
	 	\item $\frac{x_v}{x_u}\mu^N_\up$ divides $\mu$, and
	 	\item $\frac{x_v}{x_u}\mu^N_\up\in\Borel(M_\up^{k-r})$.
	 \end{itemize}
	 
	
	Since $x_s^A\mu^N_\up\mu^N_\lw=\mu$, it follows that $x_v$ must divide $\mu^N_\lw$.  Hence there is some $N'_j$, $k-r+1<j\le k$, so that $x_v$ divides $N'_j$.  Putting these together, it follows that $\frac{x_v}{x_u}N_i\in\Borel(M),\frac{x_u}{x_v}N_j\in\Borel(M)$.  Put
	\[
	\bT_{N'}=T_{\frac{x_v}{x_u}N_i}T_{\frac{x_u}{x_v}N_j}\prod_{\ell\neq i,j}T_{N_\ell}.
	\]
	Since the exponent of $x_s$ is $q$ in both $N_i$ and $\frac{x_v}{x_u}N_i$ and $q+1$ in both $N_j$ and $\frac{x_u}{x_v}N_j$, we have 
	\[
	N_i\succ_{\grevlex}\frac{x_v}{x_u}N_i\succ_{\grevlex}\frac{x_u}{x_v}N_j\succ_{\grevlex} N_j.
	\]
	It follows that $\bT_{N}\succ \bT_{N'}$, hence the directed edge $\bT_N\to\bT_{N'}$ appears in $\vec{\Gamma}_{\mu,\calB}$ and $\bT_N$ is not a sink.  Thus we may assume $\mu^N_\up=\mu_\up$ and $\mu^N_\lw=\mu_\lw$.
	
	Now, as in Discussion~\ref{disc:1}, let $\bT_\up$ and $\bT_\lw$ be the sets of variables $\{T_m:m\in\Borel(M_\up)\}$ and $\{T_n:n\in\Borel(M_\lw)\}$, respectively.  By induction on the number of variables, it follows that $\bT^{\mu_\up}_{\min}\in K[\bT_\up]_{\mu_\up}$ (respectively $\bT^{\mu_\lw}_{\min}\in K[\bT_\lw]_{\mu_\lw}$) is the unique sink of $\vec{\Gamma}_{\mu_\up,\calB_{\up}}$ (respectively $\vec{\Gamma}_{\mu_\lw,\calB_{\lw}}$). As described in Discussion~\ref{disc:1}, we may use the map
	\[
	\phi_{\up,\lw}:K[\bT_\up]_{\mu_\up}\times K[\bT_\lw]_{\mu_\lw}\to K[\bT]_\mu
	\]
	to identify $\vec{\Gamma}_{\mu_\up,\calB_{\up}}\times \vec{\Gamma}_{\mu_\lw,\calB_\lw}$ with a directed subgraph of $\vec{\Gamma}_{\mu,\calB}$ which we denote by $\vec{\Gamma}_{\up,\lw}$.  Since $T_N$ is a sink of $\vec{\Gamma}_{\mu,\calB}$ by assumption, it must also be a sink of $\vec{\Gamma}_{\up,\lw}$.  However, $\vec{\Gamma}_{\up,\lw}$ has a unique sink by induction, namely $\phi_{\up,\lw}(\bT^{\mu_\up}_{\min},\bT^{\mu_\lw}_{\min})$, which we saw in Discussion~\ref{disc:1} is the same as $\bT^\mu_{\min}$.  Thus $\bT_N=\phi_{\up,\lw}(\bT^{\mu_\up}_{\min},\bT^{\mu_\lw}_{\min})=\bT^\mu_{\min}$, as desired.
	
	Finally, $\calB(M)$ and the set of quadrics
	\[
	\{T_mT_n-T_{\mu_1}T_{\mu_2}:m,n\in\Borel(M),\bs(M,mn)=\{\mu_1,\mu_2\}\}
	\]
	have the same span as a $\K$-vector space.  This follows from the fact that $\vec{\Gamma}_{mn,\calB}$ has the unique sink $\bT^{mn}_{\min}$.
\end{proof}


\begin{corollary}[De Negri~\cite{DN99}]
	If $I$ is a principal Borel ideal, then both the toric ring $\K[I]$ and the Rees algebra $R[It]$ are Koszul, Cohen-Macaulay, and normal.  
\end{corollary}
\begin{proof}
	The Koszul property follows because $J_{\K[I]}$ has a Gr\"obner basis of quadrics by Theorem~\ref{thm:LexOrderReesAlgebraPrincipleBorel}.  The lead terms of this Gr\"obner basis are squarefree, hence it follows from~\cite[Proposition~13.15]{S96} that $K[I]$ is normal.  Hochster's well-known result then implies that $K[I]$ is Cohen-Macaulay.  These properties pass to the Rees algebra $R[It]$ by~\cite[Theorem~5.1]{HHV05}.
\end{proof}


\section{Principal $L$-Borel ideals and chordal bipartite incidence graphs}\label{sec:LBorel}
In this section we define a class of ideals containing the principal Borel ideals, which we call principal $L$-Borel ideals.  We also define an incidence condition for a collection of principal $L$-Borel ideals which is crucial for the Koszul property of the associated multi-Rees algebra, as we will see in Section~\ref{sec:MultiRees}.

\begin{definition}\label{def:LinearPoset}
	Let $U$ be a non-empty subset of $\{x_1,\ldots,x_n\}$.  Write $L_U$ for the partially ordered set (poset) on $\{x_1,\ldots,x_n\}$ defined by $x_j<_{L_U} x_i$ whenever $x_i,x_j\in U$, $i<j$, and $x_k$ is incomparable to every other variable if $x_k\notin U$.  We call a poset $L$ on $\{x_1,\ldots,x_n\}$ a \textit{linear} poset if $L=L_U$ for some $U\subset \{x_1,\ldots,x_n\}$, and we call $U$ the \textit{support} of $L_U$.
\end{definition}

Clearly a non-empty subset $U\subset\{x_1,\ldots,x_n\}$ determines a unique linear poset $L$ with support $U$.  The opposite is also true unless $L$ is the anti-chain (i.e. $L$ does not compare any variables): if $U$ is any subset consisting of a single variable then the linear poset on $U$ is the anti-chain.

\begin{definition}\label{def:LBorel}
	Suppose $L=L_U$ is a linear poset on $\x=\{x_1,\ldots,x_n\}$ and $m$ is a monomial in $S=\K[x_1,\ldots,x_n]$.  An $L$-\textit{Borel} move (respectively, reverse $L$-Borel move) on $m$ is a monomial of the form $\frac{x_i}{x_j}m$ (respectively, $\frac{x_j}{x_i}m$), where $x_j<_L x_i$.  If a set of monomials is closed under $L$-Borel moves we call it an $L$-\textit{Borel set}.  We call the set of all monomials which can be obtained from a fixed monomial $m$ by $L$-Borel moves a principal $L$-Borel set with principal $L$-Borel generator $m$; we denote such a set by $L$-$\Borel(m)$.  We call a monomial ideal generated by an $L$-Borel set (respectively, principal $L$-Borel set) an $L$-Borel ideal (respectively, principal $L$-Borel ideal).
\end{definition}

\begin{definition}\label{def:EssentialVariables}
	Suppose $I$ is an $L$-Borel ideal.  We call a variable $x_i\in K[\x]$ an \textit{essential variable} of $I$ if
	\begin{enumerate}
		\item $x_i$ divides at least one monomial in $\gens(I)$
		\item $x_i$ does not appear in every monomial of $\gens(I)$ with the same exponent.
	\end{enumerate}
\end{definition}

\begin{observation}\label{obs:EssentialVariables}
	If $I$ is $L$-Borel with essential variables $E$, and $L_E$ is the linear poset on $E$, then $I$ is also an $L_E$-Borel ideal.  Moreover, $E$ is the smallest possible support of a linear poset $L$ with respect to which $I$ is $L$-Borel.
\end{observation}

The essential variables reflect which variables actually take part in $L$-Borel moves on generators of $I$.  We illustrate this with two examples.

\begin{example}
	Suppose $U=\{x_k,x_{k+1},\ldots,x_n\}$ for some fixed integer $1\le k\le n$ and let $L=L_U$.  Then $I=L$-$\Borel(x_k^a)=\langle x_k^a\rangle$ for any $a\in\mathbb{Z}_{>0}$.  According to Definition~\ref{def:EssentialVariables}, the set of essential variables of $I$ is empty.  This is appropriate because there are no $L$-Borel moves which occur among the generators of $I$.  Thus the antichain $L_{\emptyset}$ is the poset of minimal support with respect to which $I$ is $L$-Borel.
\end{example}

\begin{example}
	Consider the ideal $I=\langle x_1^3x_4,x_1^2x_2x_4,x_1x_2^2x_4\rangle\subset \K[x_1,x_2,x_3,x_4]$.  Let $U_1=\{x_1,x_2\},$ $U_2=\{x_1,x_2,x_3\}$, and put $L_1=L_{U_1}$, $L_2=L_{U_2}$.  Then $I=\langle L_1$-$\Borel(x_1x_2^2x_4)\rangle$ and $I=\langle L_2$-$\Borel(x_1x_2^2x_4)\rangle$.  We see that the essential variables of $I$ are $\{x_1,x_2\}=U_1$.
	
	Now consider the ideal $I'=\langle x_1^3,x_1^2x_2,x_1x_2^2\rangle\subset \K[x_1,x_2,x_3,x_4]$.  Let $U_1$ and $U_2$ be as above and set $U_3=\{x_1,x_2,x_3,x_4\}$ and $L_3=L_{U_3}$.  Then $I'=\langle L_1$-$\Borel(x_1x_2^2)\rangle=\langle L_2$-$\Borel(x_1x_2^2)\rangle=\langle L_3$-$\Borel(x_1x_2^2)\rangle =\langle \Borel(x_1x_2^2)\rangle$.  The essential variables of $I'$ are again $U_1=\{x_1,x_2\}$.
\end{example}

\begin{remark}\label{rem:LBorelExt}
	Let $U\subset \{x_1,\ldots,x_n\}$, $L=L_U$, and $m\in\K[\x]$ be a monomial.  Suppose $m=m_1m_2$ where $m_1$ is a monomial in variables contained in $U$ and $m_2$ is a monomial in variables not contained in $U$.  If we consider $m_1$ as a monomial in $\K[E]$, where $E$ is the essential variables of $\langle L$-$\Borel(m)\rangle$, then $L$-$\Borel(m)=\{m_2n:n\in\Borel(m_1)\}$.  Thus $L$-Borel ideals in $\K[\x]$ are extensions of Borel ideals from a polynomial subring, possibly multiplied by an additional monomial.  We will see in Proposition~\ref{prop:multireduction} that, as far as the defining equations of Rees and multi-Rees algebras are concerned, we may assume every principal $L$-Borel ideal is simply an extension of a principal Borel ideal from $\K[E]$.
\end{remark}

\begin{remark}\label{rem:QBorel}
	We chose our notation of $L$-Borel ideals to be consistent with~\cite{FMS13}, where the more general class of $Q$-Borel ideals is introduced.  These are ideals fixed under Borel moves drawn from an arbitrary poset $Q$ on $\x$.
\end{remark}

The Koszul property for principal $L$-Borel ideals can be subtle, as the following example from~\cite{BC17} shows.

\begin{example}\label{ex:NonKoszul}
	Let $\calI=\{I_1,I_2,I_3\}$ where $I_1=\langle x_1,x_2\rangle, I_2=\langle x_1,x_3\rangle$, and $I_3=\langle x_2,x_3\rangle$.  These are principal $L$-Borel ideals on the linear posets over $E_1=\{ x_1,x_2\},$ $E_2=\{x_1,x_3\},$ and $E_3=\{x_2,x_3\}$.  The multi-fiber ring $\K[\calI\bt]$ is the subring
	\[
	\K[x_1t_1,x_2t_1,x_1t_2,x_3t_2,x_2t_3,x_3t_3]\subset \K[\x,\bt].
	\]
	It is straightforward to see that $T_{x_1t_1}T_{x_3t_2}T_{x_2t_3}-T_{x_2t_1}T_{x_1t_2}T_{x_3t_3}$ is a minimal generator of the defining ideal both of the multi-fiber ring $\K[\calI\bt]$ and the multi-Rees algebra $R[\calI\bt]$.  Thus neither the multi-fiber ring $\K[\calI\bt]$ or the multi-Rees algebra $R[\calI\bt]$ is Koszul.
\end{example}

The fiber ring in Example~\ref{ex:NonKoszul} is in fact the toric edge ring of the bipartite graph on $\{x_1,x_2,x_3\}\sqcup\{t_1,t_2,t_3\}$ with edges $\{x_1,t_1\}$, $\{x_2,t_1\}$, $\{x_1,t_2\}$, $\{x_3,t_2\}$, $\{x_2,t_3\}$, and $\{x_3,t_3\}$.  By~\cite{OH99} the edge ring of a bipartite graph is Koszul if and only if the graph is \textit{chordal bipartite}; that is, every cycle of length greater than four has a chord.  Clearly the bipartite graph associated with Example~\ref{ex:NonKoszul} is a six-cycle without a chord, so the multi-fiber ring $\K[\calI\bt]$ cannot be Koszul.  In fact, one way to interpret the main result of~\cite{OH99} is as a classification of when the fiber ring of the multi-Rees algebra $R[\calI\bt]$ is Koszul, in the case when each ideal in $\calI$ is generated by a subset of the variables of $\K[\x]$ (see Theorem~\ref{thm:OH99}).  We consider how we can put a similar condition on a collection $\calI$ of principal $L$-Borel ideals to recover the Koszul property of $R[\calI\bt]$.


\begin{definition}
	Let $G$ be a bipartite graph with vertex set $V=\x\sqcup \bt$ and edges $E\subset \x\times \bt$.  For a fixed ordering $\{x_1,\ldots,x_n\}$ of $\x$ and $\{t_1,\ldots,t_r\}$ of $\bt$, the bi-adjacency matrix of $G$ is the matrix with rows indexed by $\x$, columns indexed by $\bt$, with a $1$ in position $(x_i,t_j)$ if $\{x_i,t_j\}\in E(G)$ and a $0$ otherwise.
	
	An $\mathsf{L}$-free ordering of $V$ is an ordering $\x=\{x_1,\ldots,x_n\}$ and $\bt=\{t_1,\ldots,t_r\}$ satisfying that if $1\le h<j\le n$, $1\le u<v\le r$, and $\{x_h,t_u\},\{x_j,t_u\},\{x_j,t_v\}\in E$, then $\{x_j,t_u\}\in E$.  Equivalently, the bi-adjacency matrix of $G$ has no submatrix of the form $\begin{bmatrix} 1 & 0\\ 1 & 1\end{bmatrix}$.  (This explains the terminology $\mathsf{L}$-free.)
\end{definition}

\begin{definition}
	Let $G$ be a bipartite graph with vertex set $V=\x\sqcup \bt$ and edges $E\subset \x\times \bt$.  $G$ is a \textit{chordal bipartite} graph if every cycle of length at least six has a chord.  A \textit{chord} of a cycle is an edge of $G$ connecting non-adjacent vertices of the cycle.
\end{definition}

\begin{theorem}\cite{HKS85}\label{thm:ChordalBipartite}
	A bipartite graph $G$ is chordal bipartite if and only if its vertex set has an $\mathsf{L}$-free ordering.
\end{theorem}

\begin{remark}
	In~\cite{HKS85} and also in~\cite{OH99} the authors consider an ordering for $\x$ and $\bt$ so that the bi-adjacency matrix has no submatrix of the form $\mathsf{\Gamma}=\begin{bmatrix} 1 & 1\\ 1 & 0\end{bmatrix}$.  An $\mathsf{L}$-free ordering can easily be turned into an ordering which avoids submatrices of type $\mathsf{\Gamma}$ by reversing the order on $\x$ (since we have stipulated that the columns are labeled by $\x$).  We choose the convention of an $\mathsf{L}$-free ordering instead of a $\mathsf{\Gamma}$-free ordering because it matches best with both the Borel order on the $\x$ variables and (one of) the conventions for Ferrers diagrams -- see Example~\ref{ex:FerrersDiagrams}.
\end{remark}


\begin{definition}
	Suppose that $\calI=\{I_1,\ldots,I_r\}$ is a collection of principal $L$-Borel ideals in $\K[\x]$.  Let $E_i$ be the essential variables of $I_i$ for $i=1,\ldots,r$.  We define the \textit{essential variables incidence graph of $\calI$}, denoted $G(\calI)$ as the bipartite graph on the vertex set $\x\sqcup\bt$ with edges $\{x_i,t_j\}$ if $x_i\in E_j$.  
	
	We say an ordering of the ideals $\calI$ is $\mathsf{L}$-free if it yields an $\mathsf{L}$-free ordering of the vertices of the essential variables incidence graph of $\calI$ (we must fix the ordering of the variables $\x$ since this ordering determines what moves are considered Borel moves).
\end{definition}

\begin{example}\label{ex:LBorelFamily}
	Consider the family $\calI=\{I_1,I_2,I_3,I_4,I_5\}$ of principal $L$-Borel ideals in $\K[x_1,x_2,x_3,x_4]$:
	\[
	\begin{array}{rl}
	I_1= & \langle x_4\rangle\\
	I_2= & \langle x_3^2,x_3x_4\rangle\\
	I_3= & \langle x_2^2,x_2x_3,x_2x_4,x_3^2,x_3x_4\rangle\\
	I_4= & \langle x_1^3,x_1^2x_2,x_1^2x_3,x_1x_2^2,x_1x_2x_3\rangle\\
	I_5= & \langle x_1^3,x_1^2x_2,x_1x_2^2\rangle
	\end{array}
	\]
	The ideals $I_i$ are $L_i$-Borel (for $1\le i\le 5$) with respect to the linear posets $L_{E_i}$, where $E_1=\{x_4\}, E_2=\{x_3,x_4\}, E_3=\{x_2,x_3,x_4\}, E_4=\{x_1,x_2,x_3\},$ and $E_5=\{x_1,x_2\}$, respectively.  The sets $E_1,E_2,E_3,E_4,$ and $E_5$ are the essential variables of $I_1,I_2,I_3,I_4,$ and $I_5$, respectively.  The principal $L_i$-Borel generators are, respectively, $x_4$, $x_3x_4$, $x_3x_4$, $x_1x_2x_3$, and $x_1x_2^2$.  The essential variables incidence graph $G(\calI)$ is shown in Figure~\ref{fig:essentialvariablesincidencegraph}.  The bi-adjacency matrix of $G(\calI)$ is
	\[
	\kbordermatrix{ & t_1 & t_2 & t_3 & t_4 & t_5 \\
		x_1 & 0 & 0 & 0 & 1 & 1 \\ 
		x_2 & 0 & 0 & 1 & 1 & 1 \\
		x_3 & 0 & 1 & 1 & 1 & 0 \\
		x_4 & 1 & 1 & 1 & 0 & 0
	}.
	\]
	Since the bi-adjacency matrix is $\mathsf{L}$-free, we see that $G(\calI)$ is chordal bipartite by Theorem~\ref{thm:ChordalBipartite}.
	
	\begin{figure}
		\centering
		\begin{tikzpicture}
		\node (x1) at (0,4) {$x_1$};
		\node (x2) at (0,3) {$x_2$};
		\node (x3) at (0,2) {$x_3$};
		\node (x4) at (0,1) {$x_4$};
		
		\node (t1) at (2,4.5) {$t_1$};
		\node (t2) at (2,3.5) {$t_2$};
		\node (t3) at (2,2.5) {$t_3$};
		\node (t4) at (2,1.5) {$t_4$};
		\node (t5) at (2,.5) {$t_5$};
		
		\draw (t1)--(x4) (t2)--(x3) (t2)--(x4) (t3)--(x2) (t3)--(x3) (t3)--(x4) (t4)--(x1) (t4)--(x2) (t4)--(x3) (t5)--(x1) (t5)--(x2);
		\end{tikzpicture}
		\caption{The essential variables incidence graph for Example~\ref{ex:LBorelFamily}}
		\label{fig:essentialvariablesincidencegraph}
	\end{figure}
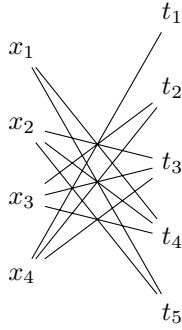

\end{example}

\begin{remark}
	Notice that $I_1,I_2,$ and $I_3$ in Example~\ref{ex:LBorelFamily} are not principal $L$-Borel with respect to any poset $L$ with support larger than $E_1,E_2,$ and $E_3$, respectively.  In contrast, $I_4$ and $I_5$ in Example~\ref{ex:LBorelFamily} are principal Borel ideals: that is, they are principal $L$-Borel with respect to the full linear poset $L$ on the set $\{x_1,x_2,x_3,x_4\}$, namely $x_4<_L x_3<_L x_2<_L x_1$.  The essential variables of $I_4$ and $I_5$ are nevertheless proper subsets of the full set of variables.
\end{remark}

\begin{example}[Producing $\mathsf{L}$-free orderings]\label{ex:FerrersDiagrams}
	Suppose $\mathcal{D}$ is a \textit{Ferrers}, \textit{skew Ferrers}, \textit{shifted Ferrers}, or \textit{shifted skew Ferrers} diagram.  We do not define these diagrams here, but we give pictures in Figure~\ref{fig:FerrersDiagrams} of diagrams of each of these types.  These are highly studied diagrams, and various ideals associated to them have quite nice properties (see for example~\cite{CN09,NR09,LS18}).  
	
	We label the rows of the diagram $\mathcal{D}$ from top to bottom by the variables $x_1,\ldots,x_n$ and the columns from left to right by $t_1,t_2,\ldots,t_r$.  We may consider the diagram $\mathcal{D}$ as a subset of the Cartesian product $\x\times\bt$ (with certain properties), and thus we can define an incidence graph on $\x\sqcup\bt$ with edges $\{x_i,t_j\}$ if $(x_i,t_j)\in\mathcal{D}$.  The bi-adjacency matrix of this incidence graph is easily seen from the diagram $\mathcal{D}$: the $(x_i,t_j)$ entry of the adjacency matrix is a $1$ if a box is present and a $0$ if a box is not present.
	
	It is straigtforward to see that this ordering of the vertices of the incidence graph of a Ferrers, skew Ferrers, or shifted skew Ferrers diagram will be an $\mathsf{L}$-free ordering.  For instance, the incidence graph of the skew Ferrers diagram in Figure~\ref{fig:FerrersDiagrams} is the essential variables incidence graph for the collection of ideals $\calI$ in Example~\ref{ex:LBorelFamily}.
	
	\begin{figure}
		\centering
		
		\begin{tikzpicture}[scale=.8]
		\draw[thick,black] (0,0)--(5,0) (0,-1)--(5,-1) (0,-2)--(5,-2) (0,-3)--(4,-3) (0,-4)--(3,-4);
		\draw[thick,black] (0,0)--(0,-4) (1,0)--(1,-4) (2,0)--(2,-4) (3,0)--(3,-4) (4,0)--(4,-3)  (5,0)--(5,-2);
		\node at (.5,.5){$t_1$};
		\node at (1.5,.5){$t_2$};
		\node at (2.5,.5){$t_3$};
		\node at (3.5,.5){$t_4$};
		\node at (4.5,.5){$t_5$};
		\node at (-.5,-.5){$x_1$};
		\node at (-.5,-1.5){$x_2$};
		\node at (-.5,-2.5){$x_3$};
		\node at (-.5,-3.5){$x_4$};
		\end{tikzpicture}
		
		\vspace{15 pt}
		
		\begin{tikzpicture}[scale=.6]
		\draw[thick,black] (3,0)--(5,0) (2,-1)--(5,-1) (1,-2)--(5,-2) (0,-3)--(4,-3) (0,-4)--(3,-4);
		\draw[thick,black] (0,-3)--(0,-4) (1,-2)--(1,-4) (2,-1)--(2,-4) (3,0)--(3,-4) (4,0)--(4,-3)  (5,0)--(5,-2);
		\end{tikzpicture}
		\hspace{10 pt}
		\begin{tikzpicture}[scale=.6]
		\draw[thick,black] (0,0)--(5,0) (0,-1)--(5,-1) (1,-2)--(5,-2) (2,-3)--(4,-3) (3,-4)--(4,-4);
		\draw[thick,black] (0,0)--(0,-1) (1,0)--(1,-2) (2,0)--(2,-3) (3,0)--(3,-4) (4,0)--(4,-4)  (5,0)--(5,-2);
		\end{tikzpicture}
		\hspace{10 pt}
		\begin{tikzpicture}[scale=.6]
		\draw[thick,black] (4,0)--(5,0) (2,-1)--(5,-1) (2,-2)--(5,-2) (2,-3)--(4,-3) (3,-4)--(4,-4);
		\draw[thick,black] (2,-1)--(2,-3) (3,-1)--(3,-4) (4,0)--(4,-4)  (5,0)--(5,-2);
		\end{tikzpicture}
		
		\caption{Top: a labeled Ferrers diagram.  Bottom (left to right): a skew Ferrers, shifted Ferrers, and shifted skew Ferrers diagram}
		\label{fig:FerrersDiagrams}
	\end{figure}
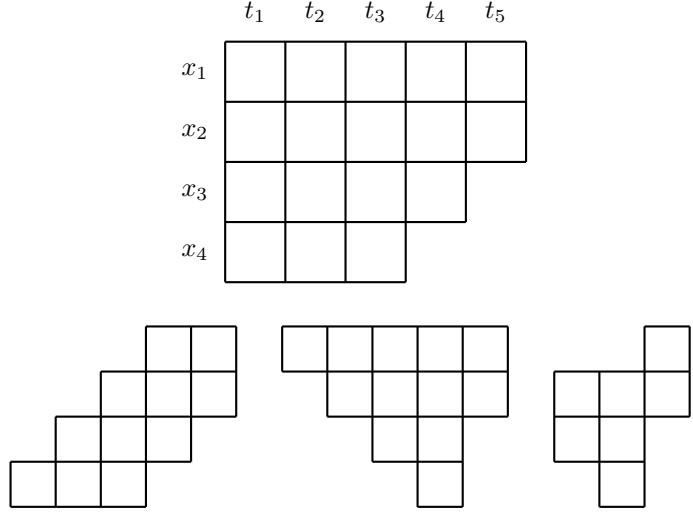
	
\end{example}

An important special case is when $\calI$ is a collection of principal Borel ideals.

\begin{proposition}\label{prop:Borel}
	Suppose $\calI=\{I_1,\ldots,I_r\}$ is a collection of principal Borel ideals of $\K[\x]$.  Then $\calI$ admits an $\mathsf{L}$-free ordering.
\end{proposition}
\begin{proof}
	Let $E_i$ denote the set of essential variables of $I_i$ for $i=1,\ldots,r$.  Since $I_i$ is Borel, $E_i$ must either be empty or have the form $\{x_1,\ldots,x_k\}$ for some $2\le k\le n$ for every $1\le i\le r$.  If necessary, re-index the ideals $I_i$ so that $|E_1|\ge |E_2|\ge \cdots \ge |E_r|$.  Then the corresponding ordering on the vertices of the essential variables incidence graph $G(\calI)$ is clearly $\mathsf{L}$-free, hence so is $\calI$.  In fact, the bi-adjacency matrix of $G(\calI)$ (with this re-ordering) corresponds to a Ferrers diagram (see Example~\ref{ex:FerrersDiagrams}).
\end{proof}

\section{Koszul multi-Rees algebras of principal $L$-Borel ideals}\label{sec:MultiRees}

We begin this section by re-stating the main result of~\cite{OH99} in our context, as a statement about $L$-Borel ideals generated in degree one.

\begin{theorem}\cite{OH99}\label{thm:OH99}
	Let $\calI=\{I_1,\ldots,I_r\}$ be a collection of $L$-Borel ideals generated in degree one.  Then the multi-fiber ring $\K[\calI\bt]$ is Koszul if and only if the essential variables incidence graph $G(\calI)$ is chordal bipartite.
\end{theorem}

Our goal in this section is to generalize Theorem~\ref{thm:OH99} as much as possible to the multi-Rees algebra and multi-fiber ring of an arbitrary collection of $L$-Borel ideals.  (We cannot completely generalize Theorem~\ref{thm:OH99} -- one difficulty is that we are not free to re-order the $\x$ variables as this could destroy the $L$-Borel property for the ideals in $\calI$.)  We focus on the multi-Rees algebra, however slight modifications in the proofs will yield similar statements for the multi-fiber ring, as we indicate later.  Before proceeding, we show that we can assume the $L$-Borel ideals of $\calI$ are simply extensions of principal Borel ideals from polynomial subrings (see Remark~\ref{rem:LBorelExt}).

\begin{proposition}\label{prop:multireduction}
	Let $M_1,\ldots,M_r$ be a collection of monomials in $\K[\x]$.  Let $L_1,\ldots,L_r$ be linear posets on $\x$, and put $I_i=L_i$-$\Borel(M_i)$ for $i=1,\ldots,r$.  Suppose that, for $i=1,\ldots,r$, $M_i$ factors as $M_i=M'_im_i$, where $M'_i$ is divisible only by variables in the support of the essential variables $E_i$ and $m_i$ is divisible only by variables not in the support of $E_i$.  Put $I'_i=L_i$-$\Borel(M_i)$ for $i=1,\ldots,r$ and $\calI'=\{I'_1,\ldots,I'_r\}$.  Then $R[\calI\bt]\cong R[\calI'\bt]$ and $\K[\calI\bt]\cong \K[\calI'\bt]$.
\end{proposition}
\begin{proof}
	We define a map $\phi:R[\calI'\bt]\to R[\calI\bt]$ on generators by $\phi(x_i)=x_i$ for $i=1,\ldots,n$ and $\phi(nt_i)=m_int_i$ when $n\in \gens(I'_i)$, and extend linearly.  This map is clearly surjective.  Since the powers of $m_i$ can be read off the powers of $t_i$, it is also straightforward to see that $\phi$ is injectve.  Hence $R[\calI\bt]\cong R[\calI'\bt]$.  A similar map shows $\K[\calI\bt]\cong \K[\calI'\bt]$.
\end{proof}

Henceforth we only consider $\calI=\{I_1,\ldots,I_r\}$, where $I_i=L_i$-$\Borel(M_i)$ and $M_i$ is only divisible by essential variables of $I_i$.  Equivalently, $I_i$ is the extension of the ideal $\Borel(M_i)$ from the polynomial subring $\K[E_i]$ for $i=1,\ldots,r$, where $E_i$ is the set of essential variables of $I_i$.

\begin{convention}[Monomial order for $\K\left\lbrack\x,\bT\right\rbrack$ ]\label{conv:MonomialOrder}
	Suppose $\calI=\{I_1,\ldots,I_r\}$ is a collection of principal $L$-Borel ideals of $\K[\x]$ and that the given ordering of $\calI$ is an $\mathsf{L}$-free ordering.  We order the variables of $\K[\x,\bT]$ by $T_{mt_i}\prec T_{nt_j}$ if $i>j$ or $i=j$ and $m\prec_{\grevlex} n$.  Furthermore we make $T_{mt_i}\succ x_j$ for any variable $T_{mt_i}\in\bT$ and $x_j\in\x$.  On top of this ordering of the variables of $\K[\x,\bT]$ we put the lexicographic order.
\end{convention}

\begin{definition}\label{def:MultiSink}
	Let $M_1,\ldots,M_r$ be a collection of monomials in $\K[\x]$.  Let $L_1,\ldots,L_r$ be linear posets on $\x$, and put $I_i=L_i$-$\Borel(M_i)$ for $i=1,\ldots,r$.  We assume that $L_i=L_{E_i}$, the linear poset associated to the essential variables of $I_i$, for $i=1,\ldots,r$.  We further assume, using Proposition~\ref{prop:multireduction} if necessary, that every variable dividing $M_i$ is in the support of $E_i$, so $I_i$ is the extension of a principal Borel ideal from the subring $\K[E_i]$ for $i=1,\ldots,r$.  Suppose $\mu\bt^{\bbeta}=\mu t_1^{\beta_1}\cdots t_r^{\beta_r}$ is a monomial in $\K[\x,\bt]$.  Define $\mu_1,\ldots,\mu_r$ inductively as follows.  If $\beta_1=0$, then set $\mu_1=1$.  Otherwise, if there is a monomial in $L_1$-$\Borel(M_1^{\beta_1})$ which divides $\mu$, let $\mu_1$ be the smallest monomial in $L_1$-$\Borel(M_1^{\beta_1})$ under Borel order dividing $\mu$ (which exists by Lemma~\ref{lem:BorelMultidegreeDivision}).  If there is no monomial in $L_1$-$\Borel(M_1^{\beta_1})$ which divides $\mu$, then we stop the procedure and say $\mu_i$ is not defined for $i=1,\ldots,r$.  Now we define $\mu_{\ell+1}$ from $\mu_{\ell}$.  First, if $\beta_{\ell+1}=0$, then $\mu_{\ell+1}=1$.  Otherwise, if there is a monomial in $L_{\ell+1}$-$\Borel(M_{\ell+1}^{\beta_{\ell+1}})$ which divides $\mu/(\mu_1\cdots\mu_{\ell})$, then we set $\mu_{\ell+1}$ to be the smallest such monomial under Borel order.  If there is no monomial in $L_{\ell+1}$-$\Borel(M_{\ell+1}^{\beta_{\ell+1}})$ which divides $\mu/(\mu_1\cdots\mu_{\ell})$, then we stop the procedure and say $\mu_{k}$ is not defined for $\ell<k\le r$.  If $\mu_1,\ldots,\mu_r$ are all defined then, for $i=1,\ldots,r$, we define
	\[
	\bT^{\mu_it_i^{\beta_i}}_{\min}:=\prod_{j=1}^{\beta_i} T_{N_jt_i},
	\]
	where $\{N_1,\ldots,N_{\beta_i}\}=\bs(M_i,\mu_i)$.  (The Borel sort algorithm should be carried out in the polynomial ring $\K[E_i]$, where $E_i$ is the set of essential variables of $I_i$, for $i=1,\ldots,k$.)  Finally (again assuming $\mu_1,\ldots,\mu_r$ are defined) we put $\nu=\mu/(\mu_1\cdots\mu_r)$ and define
	\[
	\bT^{\mu\bt^{\bbeta}}_{\min}:=\nu \bT^{\mu_1t_1^{\beta_1}}_{\min} \cdots \bT^{\mu_rt_r^{\beta_r}}_{\min}.
	\]
	If any of the monomials $\mu_1,\cdots,\mu_r$ are not defined, then we say $\bT^{\mu\bt^{\bbeta}}_{\min}$ is not defined.
\end{definition}

\begin{remark}
	By construction, if $\bT^{\mu\bt^{\bbeta}}_{\min}$ is defined then it is the lexicographically smallest monomial in $\K[\x,\bT]_{\mu\bt^{\bbeta}}$.
\end{remark}

The procedure outlined in Definition~\ref{def:MultiSink} often fails, however we will show that it never fails if $\calI$ is an $\mathsf{L}$-free ordered collection of $L$-Borel ideals.

\begin{example}
	Let $I_1=\langle x_1,x_2\rangle, I_2=\langle x_1,x_3\rangle,$ and $I_3=\langle x_2,x_3\rangle$, and $\calI=\{I_1,I_2,I_3\}$ be the collection of $L$-Borel ideals from Example~\ref{ex:NonKoszul}.  Let $\mu\bt^{\bbeta}=x_1x_2x_3t_1t_2t_3$, so $\mu=x_1x_2x_3$.  Then the minimal monomial in $I_1$ under Borel order dividing $\mu$ is $x_2$.  So $\mu_1=x_2$.  Likewise the minimal monomial in $I_2$ dividing $\mu/\mu_1=x_1x_3$ is $x_3$, so $\mu_2=x_3$.  However it is then impossible to define $\mu_3$.  Notice that there is no $\mathsf{L}$-free ordering of $\calI$ in this case.
\end{example}

\begin{example}
	Let $\calI$ be the collection of $L$-Borel ideals from Example~\ref{ex:LBorelFamily}, and put $\mu\bt^{\bbeta}=x_1^6x_2^9x_3^6x_4^4t_1t_2^2t_3^2t_4^2t_5^2$.  Then we have $\mu_1=x_4$, $\mu_2=x_3^2x_4^2$, $\mu_3=x_3^3x_4$, $\mu_4=x_1^2x_2^3x_3$, $\mu_5=x_1^2x_2^4$, and $\nu=x_1^2x_2^2$.  We then compute
	\[
	\bT^{\mu\bt^{\bbeta}}_{\min}=x_1^2x_2^2T_{x_4t_1}T^2_{x_3x_4t_2}T_{x_3^2t_3}T_{x_3x_4t_3}T_{x_1x_2^2t_4}T_{x_1x_2x_3t_4}T^2_{x_1x_2^2t_5}.
	\]
\end{example}

Now we define the set of quadrics which we will show is a Gr\"obner basis for the defining equations of the multi-Rees algebra.  Suppose that $\calI$ is an $\mathsf{L}$-free ordered collection of $L$-Borel ideals.  We consider quadrics in the defining ideal of the multi-Rees algebra $R[\calI\bt]$ of three types.  First, we have quadrics of \textit{symmetric type}; we name them symmetric type because they are relations on the symmetric algebra.  These are quadrics of the form
\[
x_sT_{mt_i}-x_tT_{\frac{x_s}{x_t}mt_i}
\]
where $m\in I_i$ and $\frac{x_s}{x_t}m$ is an $L_i$-Borel move on $m$.

Then we have quadrics of \textit{fiber type}, which are relations on the defining ideal of the multi-fiber ring.  We split these further into \textit{principal} and \textit{bi-principal} fiber type.  A quadric of principal fiber type has the form
\[
T_{mt_i} T_{nt_i}-T_{\frac{x_s}{x_t}mt_i}T_{\frac{x_t}{x_s}nt_i},
\]
where $m,n\in I_i$, $\frac{x_s}{x_t}m$ is an $L_i$-Borel move on $m$, and $\frac{x_t}{x_s}n$ is a reverse $L_i$-Borel move on $n$.
A quadric of bi-principal fiber type has the form
\[
T_{mt_i} T_{nt_j}-T_{\frac{x_t}{x_s}mt_i}T_{\frac{x_s}{x_t}nt_j},
\]
where $m\in I_i,n\in I_j$, $i<j$, $\frac{x_t}{x_s}m$ is a reverse $L_i$-Borel move on $m$, and $\frac{x_s}{x_t}n$ is an $L_j$-Borel move on $n$.

\begin{theorem}\label{thm:MultiSink}
	Let $M_1,\ldots,M_r$ be a collection of monomials in $\K[\x]$.  Let $L_1,\ldots,L_r$ be linear posets on $\x$, and put $I_i=L_i$-$\Borel(M_i)$ for $i=1,\ldots,r$.  Let $\calI=\{I_1,\ldots,I_r\}$ be an $\mathsf{L}$-free ordering and $\mathcal{B}\subset \K[\x,\bT]$ the set of quadrics of symmetric and fiber type. 
	
	Then, for any $\mu\bt^{\bbeta}=\mu t_1^{\beta_1}\cdots t_r^{\beta_r}\in\K[\x,\bt]$ for which $\K[\x,\bT]_{\mu\bt^{\bbeta}}$ is non-empty, the monomial $\bT^{\mu\bt^{\bbeta}}_{\min}\in \K[\x,\bT]_{\mu\bt^{\bbeta}}$ from Definition~\ref{def:MultiSink} is the unique sink of $\vec{\Gamma}_{\mu\bt^{\bbeta},\calB}$.  In particular, the quadrics of symmetric and fiber type form a quadratic Gr\"obner basis for $J_{R[\calI\bt]}$ with respect to the monomial order in Convention~\ref{conv:MonomialOrder}.
\end{theorem}

\begin{proof}
	We induct on the size of the support of $\bbeta$; that is, the number of integers $\beta_i$ which are non-zero.  If $\bbeta=\mathbf{0}$, then $\mu'=\mu=\bT^{\mu\bt^{\bbeta}}_{\min}$ is the only monomial in $\K[\x,\bT]_{\mu\bt^{\bbeta}}$, so the theorem is trivially satisfied.
	
	Now suppose that $\bt^{\bbeta}$ is divisible by $t_u$ for some $u\ge 1$.  We assume $u$ is the least integer so that $t_u$ divides $\bt^{\bbeta}$.  If $u=1$, then $\mu_u=\mu_1$ in the notation of Definition~\ref{def:MultiSink}.  Otherwise $\mu_1=\mu_2=\cdots=\mu_{u-1}=1$, and $\mu_u$ is the smallest monomial in $\Borel(M_u^{\beta_u})$ which divides $\mu$ (since we assume $\K[\x,\bT]_{\mu\bt^{\bbeta}}$ is non-empty there must be at least one monomial in $\Borel(M_u^{\beta_u})$ which divides $\mu$).  Now suppose $\bT'=\nu'\prod_{i=u}^{k}\prod_{s=1}^{\beta_i} T_{m_s^{(i)}t_i}$ is a monomial in $\K[\x,\bT]_{\mu\bt^{\bbeta}}$, where $\nu'\in\K[\x]$ and the superscript $(i)$ records that $m_s^{(i)}\in \Borel(M_i)$.  Assume further that $\bT'$ is a sink.  Put $\bT'_u:=\prod_{s=1}^{\beta_u}T_{m_s^{(u)}t_u}$ and $\mu'_u:=\prod_{s=1}^{\beta_u} m_s^{(u)}$.
	
	First, assume that $\mu'_u=\mu_u$.  Then, since we assume $\bT'$ is a sink, we must have $\bT'_u=\bT^{\mu_ut_u^{\beta_u}}_{\min}$ by Theorem~\ref{thm:LexOrderReesAlgebraPrincipleBorel} (here we use only the fiber quadrics of principal type).  Since we assume $\bT'$ is a sink, $\bT'/\bT'_u$ must also be a sink of $\vec{\Gamma}_{\mu\bt^{\bbeta}/(\mu_u t_u^{\beta_u}),\calB}$.  By induction on the size of the support of $\bbeta$, $\bT'/\bT'_u=\bT^{\mu\bt^{\bbeta}/(\mu_u t_u^{\beta_u})}_{\min}$ hence $\bT'=\bT^{\mu\bt^{\bbeta}}_{\min}$.
	
	Now assume $\mu'_u\neq \mu_u$.  We necessarily have that $\mu'_u$ divides $\mu$ and $\mu'_u\in\Borel(M_u^{\beta_u})$.  It follows from the definition of $\mu_u$ and Lemma~\ref{lem:BorelMultidegreeDivision} that $\mu'_u\in L_u$-$\Borel(\mu_u)$ and there is a reverse $L_u$-Borel move $\frac{x_j}{x_h}\mu'_u\in L_u$-$\Borel(\mu_u)$ so that $\frac{x_j}{x_h}\mu'_u$ divides $\mu$.  Since $\prod_{s=1}^{\beta_i} m^{(i)}_s=\mu'_i$, the same reverse Borel move can be applied to one of the factors $m^{(i)}_s$, preserving that $\frac{x_j}{x_h}m^{(i)}_s\in L_u$-$\Borel(M_u)$.  Without loss assume $s=1$, so $\frac{x_j}{x_h}m^{(i)}_1\in L_u$-$\Borel(M_u)$.
	
	We consider two cases.  First, if $x_j$ divides $\nu$, then
	\[
	\bT''=\frac{x_h}{x_j}\bT'\cdot T_{\frac{x_j}{x_h}m^{(u)}_1t_u}/T_{m^{(u)}_1t_u}\in \K[\x,\bt]_{\mu\bt^{\bbeta}},
	\]
	$\bT''\prec \bT'$, and $\bT'-\bT''$ is divisible by the quadric 
	\[
	x_jT_{m^{(u)}_1t_u}-x_hT_{\frac{x_j}{x_h}m^{(u)}_1t_u}
	\]
	of symmetric type.  Hence $\bT'\rightarrow \bT''$ is a directed edge of $\vec{\Gamma}_{\mu\bt^{\bbeta},\calB}$ and $\bT'$ is not a sink, contrary to assumption.
	
	Now suppose $x_j$ does not divide $\nu$.  Since $\frac{x_j}{x_h}\mu'_u$ divides $\mu$, it follows that $x_j$ divides $\mu/\mu'_u$.  Thus there is some index $v>u$ and some monomial $n\in\Borel(M_v)$ so that $x_j$ divides $n$ and $T_{nt_v}$ divides $\bT'$.  Since $\calI$ is an $\mathsf{L}$-free ordering, $x_h$ must also be an essential variable of $I_v$ and thus $\frac{x_h}{x_j}n$ is an $L_v$-Borel move on $n$.  Thus
	\[
	\bT''=\bT'\cdot\frac{T_{x_j/x_hm^{(u)}_1t_u}}{T_{m^{(u)}_1t_u}}\cdot\frac{T_{x_h/x_j nt_v}}{T_{nt_v}}\in \K[\x,\bt]_{\mu\bt^{\bbeta}},
	\]
	$\bT''\prec \bT'$, and $\bT'-\bT''$ is divisible by the quadric
	\[
	T_{{m^{(u)}_1t_u}}T_{nt_v}-T_{x_j/x_hm^{(u)}_1t_u}T_{x_h/x_j nt_v}
	\]
	of bi-principal fiber type.  Hence $\bT'\rightarrow \bT''$ is a directed edge of $\vec{\Gamma}_{\mu\bt^{\bbeta},\calB}$ and $\bT'$ is not a sink, contrary to assumption.
	
	We conclude that if $\bT'$ is a sink, then $\bT'=\bT^{\mu\bt^{\bbeta}}_{\min}$.  Since $\vec{\Gamma}_{\mu\bt^{\bbeta},\calB}$ must have at least one sink, $\bT^{\mu\bt^{\bbeta}}_{\min}$ is the unique sink.
\end{proof}

Although we have focused on the defining equations of the multi-Rees algebra, a similar statement is true for the defining equations of the multi-fiber ring of the multi-Rees algebra.  We omit the proof as it is essentially a repeat of the proof of Theorem~\ref{thm:MultiSink} except we have no need of the quadrics of symmetric type.

\begin{theorem}\label{thm:MultiFiber}
	Suppose $\calI=\{I_1,\ldots,I_r\}$ is an $\mathsf{L}$-free ordered collection of principal $L$-Borel ideals.  Then the quadrics of fiber type form a quadratic Gr\"obner basis for the defining equations $J_{\K[\calI\bt]}\subset \K[\bT]$ for the multi-fiber ring $\K[\calI\bt]$.  We take the monomial order on $\K[\bT]$ to be the monomial order induced on $\K[\bT]$ as a subring of $\K[\x,\bT]$, where the latter is given the monomial order of Convention~\ref{conv:MonomialOrder}. 
\end{theorem}

\begin{corollary}\label{cor:Koszul}
	Let $\calI$ be a collection of principal $L$-Borel ideals so that $\calI$ admits an $\mathsf{L}$-free ordering.  Then $R[\calI\bt]$ and $\K[\calI\bt]$ are Koszul, Cohen-Macaulay, and normal.
\end{corollary}
\begin{proof}
	If necessary, re-order $\calI$ so that it is $\mathsf{L}$-free.  Applying Theorem~\ref{thm:MultiSink}, $J_{R[\calI\bt]}$ has a Gr\"obner basis of quadrics, hence $R[\calI\bt]$ is Koszul.  Furthermore, the leading terms of this Gr\"obner basis under the monomial order from Convention~\ref{conv:MonomialOrder} are squarefree.  Hence it follows from~\cite[Proposition~13.15]{S96} that $R[\calI\bt]$ is normal and thus by Hochster's well known result, $R[\calI\bt]$ is Cohen-Macaulay.   For the multi-fiber ring $\K[\calI\bt]$, first apply Theorem~\ref{thm:MultiFiber} and then use the same argument.
\end{proof}

\begin{corollary}\label{cor:PrincipalBorel}
	Let $\calI$ be a collection of principal Borel ideals.  Then $R[\calI\bt]$ and $\K[\calI\bt]$ are Koszul, Cohen-Macaulay, and normal.
\end{corollary}
\begin{proof}
	This is immediate from Corollary~\ref{cor:Koszul} and Proposition~\ref{prop:Borel}.
\end{proof}

\begin{corollary}\label{cor:bipperm}
	Let $\calI$ be a collection of principal $L$-Borel ideals so that each ideal of $\calI$ is also principal $L$-Borel under any permutation of $\x$.  If the essential variables incidence graph of $\calI$ is chordal bipartite then $R[\calI\bt]$ and $\K[\calI\bt]$ are Koszul, Cohen-Macaulay, and normal.
\end{corollary}
\begin{proof}
	Since $G(\calI)$ is chordal bipartite, there is an $\mathsf{L}$-free ordering of its vertices.  Fix this ordering.  With respect to this re-ordering of the variables $\x$, every ideal in $\mathcal{I}$ is still principal $L$-Borel by assumption (the linear ordering has changed according to the permutation applied to $\x$).  Hence we now have an $\mathsf{L}$-free collection of principal $L$-Borel ideals, and we can apply Corollary~\ref{cor:Koszul}.
\end{proof}

\begin{corollary}
	Suppose that $\calI$ is a collection of ideals, each of which is a power of an ideal generated by variables.  If the essential variables graph of $\calI$ is chordal bipartite then $R[\calI\bt]$ and $\K[\calI\bt]$ are Koszul, Cohen-Macaulay, and normal.
\end{corollary}
\begin{proof}
	Let $I=\{x_i:x_i\in E\}$ for some $E\subset \x$.  Then $I^k$ is principal $L_E$-Borel for any $k$ and any ordering of the variables $\x$.  The result now follows from Corollary~\ref{cor:bipperm}.
\end{proof}

\section{Concluding Remarks and Questions}\label{sec:Conclusion}
\begin{remark}
Algorithm~\ref{alg:BS} can be modified to produce unique sinks in the case that $\K[\x]$ has lexicographic instead of graded reverse lexicographic order.  Moreover, there is a much simpler algorithm than Algorithm~\ref{alg:BS} which produces the unique sink for $\vec{\Gamma}_{\mu,\calB}$ if \textit{graded reverse lexicographic order} is used on $\K[\bT]$ instead of lexicographic order.  This comes at the cost of leading terms which are no longer squarefree.
\end{remark}

\begin{question}
	If $I$ is a principal Borel ideal, is the set of binomials indicated in~\eqref{eq:ReesAlgebraGB} a \textit{universal} Gr\"obner basis for the toric ideal $J_{\K[I]}$?
\end{question}

\begin{remark}
	It follows from Corollary~\ref{cor:Koszul} that the multi-Rees algebra of a collection of principal $L$-Borel ideals which admits an $\mathsf{L}$-free ordering is Koszul.  We do not know if this is a necessary condition for the multi-Rees algebra of a collection of principal $L$-Borel ideals to be Koszul.
\end{remark}

\begin{remark}
	If $\calI=\{I_1,\ldots,I_k\}$ is an arbitrary collection of principal $L$-Borel ideals, then the multi-Rees algebra $R[\calI\bt]$ is always Cohen-Macaulay and normal. We can see this as follows. From~\cite[Proposition~2.9]{FMS13} we obtain that principal $L$-Borel ideals are polymatroidal.  The multi-Rees algebra of polymatroidal ideals is Cohen-Macaulay and normal by~\cite[Theorem~5.4]{BC17}. 
	Since principal $Q$-Borel ideals (see Remark~\ref{rem:QBorel}) are also polymatroidal, the multi-Rees algebra of a collection of principal $Q$-Borel ideals is also Cohen-Macaulay and normal.  $Q$-Borel ideals are introduced in~\cite{FMS13}; they interpolate between arbitrary monomial ideals and Borel ideals. 
\end{remark}


\begin{question}\label{ques:PrincipalQBorel}
	Does the defining equations of the Rees algebra of a principal $Q$-Borel ideal have a Gr\"obner basis of quadrics?
\end{question}

\begin{remark}\label{rem:WhiteConjecture}
	Since principal $Q$-Borel ideals are polymatroidal, Question~\ref{ques:PrincipalQBorel} is a special case of a question of Herzog and Hibi -- namely whether the toric ideal of a polymatroidal ideal has a Gr\"obner basis of quadrics~\cite{HH02}.  This in turn is an extension of a conjecture of White that the base ring of a matroid has a defining ideal generated by quadrics~\cite{W80}.
\end{remark}


\begin{question}
	Suppose we are given a collection $\calI=\{I_1,\ldots,I_r\}$ of monomial ideals in $\K[\x]$ so that $J_{\K[I_it]}$ has a Gr\"obner basis of quadrics with respect to a monomial order on $\K[\x]$ which is fixed for $i=1,\ldots,r$.  Suppose additionally that $\calI$ is a subset of a sufficiently combinatorial family of ideals, such as Lex segment ideals, ideals of Veronese type, polymatroidal ideals, etc.  Is there an appropriate incidence condition -- perhaps depending on the family of ideals of which $\calI$ is a subset -- so that $J_{\K[\calI\bt]}$ and $J_{R[\calI\bt]}$ also have a Gr\"obner basis of quadrics?  (If $\calI$ is a collection of principal $L$-Borel ideals then the appropriate incidence condition is that $\calI$ is $\mathsf{L}$-free.)
\end{question}




\begin{thebibliography}{10}
	
	\bibitem{Blasiak08}
	Jonah Blasiak.
	\newblock The toric ideal of a graphic matroid is generated by quadrics.
	\newblock {\em Combinatorica}, 28(3):283--297, 2008.
	
	\bibitem{BC17}
	Winfried Bruns and Aldo Conca.
	\newblock Linear resolutions of powers and products.
	\newblock In {\em Singularities and computer algebra}, pages 47--69. Springer,
	Cham, 2017.
	
	\bibitem{BC17b}
	Winfried Bruns and Aldo Conca.
	\newblock Products of {B}orel fixed ideals of maximal minors.
	\newblock {\em Adv. in Appl. Math.}, 91:1--23, 2017.
	
	\bibitem{CDR13}
	Aldo Conca, Emanuela De~Negri, and Maria~Evelina Rossi.
	\newblock Koszul algebras and regularity.
	\newblock In {\em Commutative algebra}, pages 285--315. Springer, New York,
	2013.
	
	\bibitem{CN09}
	Alberto Corso and Uwe Nagel.
	\newblock Monomial and toric ideals associated to {F}errers graphs.
	\newblock {\em Trans. Amer. Math. Soc.}, 361(3):1371--1395, 2009.
	
	\bibitem{CLS19}
	David~A. Cox, Kuei-Nuan Lin, and Gabriel Sosa.
	\newblock Multi-{R}ees algebras and toric dynamical systems.
	\newblock {\em Proc. Amer. Math. Soc.}, 147(11):4605--4616, 2019.
	
	\bibitem{DN99}
	Emanuela De~Negri.
	\newblock Toric rings generated by special stable sets of monomials.
	\newblock {\em Math. Nachr.}, 203:31--45, 1999.
	
	\bibitem{DFMSS19}
	Michael DiPasquale, Christopher~A. Francisco, Jeffrey Mermin, Jay Schweig, and
	Gabriel Sosa.
	\newblock The {R}ees algebra of a two-{B}orel ideal is {K}oszul.
	\newblock {\em Proc. Amer. Math. Soc.}, 147(2):467--479, 2019.
	
	\bibitem{EisenbudHunekeUlrich03}
	David Eisenbud, Craig Huneke, and Bernd Ulrich.
	\newblock What is the {R}ees algebra of a module?
	\newblock {\em Proc. Amer. Math. Soc.}, 131(3):701--708, 2003.
	
	\bibitem{FMS11}
	Christopher~A. Francisco, Jeffrey Mermin, and Jay Schweig.
	\newblock Borel generators.
	\newblock {\em J. Algebra}, 332:522--542, 2011.
	
	\bibitem{FMS13}
	Christopher~A. Francisco, Jeffrey Mermin, and Jay Schweig.
	\newblock Generalizing the {B}orel property.
	\newblock {\em J. Lond. Math. Soc. (2)}, 87(3):724--740, 2013.
	
	\bibitem{M2}
	Daniel~R. Grayson and Michael~E. Stillman.
	\newblock Macaulay2, a software system for research in algebraic geometry.
	\newblock Available at \url{http://www.math.uiuc.edu/Macaulay2/}.
	
	\bibitem{HH02}
	J\"{u}rgen Herzog and Takayuki Hibi.
	\newblock Discrete polymatroids.
	\newblock {\em J. Algebraic Combin.}, 16(3):239--268 (2003), 2002.
	
	\bibitem{HHO18}
	J\"{u}rgen Herzog, Takayuki Hibi, and Hidefumi Ohsugi.
	\newblock {\em Binomial ideals}, volume 279 of {\em Graduate Texts in
		Mathematics}.
	\newblock Springer, Cham, 2018.
	
	\bibitem{HHV05}
	J\"{u}rgen Herzog, Takayuki Hibi, and Marius Vladoiu.
	\newblock Ideals of fiber type and polymatroids.
	\newblock {\em Osaka J. Math.}, 42(4):807--829, 2005.
	
	\bibitem{HKS85}
	A.~J. Hoffman, A.~W.~J. Kolen, and M.~Sakarovitch.
	\newblock Totally-balanced and greedy matrices.
	\newblock {\em SIAM J. Algebraic Discrete Methods}, 6(4):721--730, 1985.
	
	\bibitem{Jb18}
	Babak Jabarnejad.
	\newblock Equations defining the multi-{R}ees algebras of powers of an ideal.
	\newblock {\em J. Pure Appl. Algebra}, 222(7):1906--1910, 2018.
	
	\bibitem{LP14}
	Kuei-Nuan Lin and Claudia Polini.
	\newblock Rees algebras of truncations of complete intersections.
	\newblock {\em J. Algebra}, 410:36--52, 2014.
	
	\bibitem{LS18}
	Kuei-Nuan Lin and Yi-Huang Shen.
	\newblock Koszul blowup algebras associated to three-dimensional {F}errers
	diagrams.
	\newblock {\em J. Algebra}, 514:219--253, 2018.
	
	\bibitem{MS05}
	Ezra Miller and Bernd Sturmfels.
	\newblock {\em Combinatorial commutative algebra}, volume 227 of {\em Graduate
		Texts in Mathematics}.
	\newblock Springer-Verlag, New York, 2005.
	
	\bibitem{NR09}
	Uwe Nagel and Victor Reiner.
	\newblock Betti numbers of monomial ideals and shifted skew shapes.
	\newblock {\em Electron. J. Combin.}, 16(2, Special volume in honor of Anders
	Bj\"{o}rner):Research Paper 3, 59, 2009.
	
	\bibitem{OH99}
	Hidefumi Ohsugi and Takayuki Hibi.
	\newblock Koszul bipartite graphs.
	\newblock {\em Adv. in Appl. Math.}, 22(1):25--28, 1999.
	
	\bibitem{R99}
	J.~Ribbe.
	\newblock On the defining equations of multi-graded rings.
	\newblock {\em Comm. Algebra}, 27(3):1393--1402, 1999.
	
	\bibitem{Schweig11}
	Jay Schweig.
	\newblock Toric ideals of lattice path matroids and polymatroids.
	\newblock {\em J. Pure Appl. Algebra}, 215(11):2660--2665, 2011.
	
	\bibitem{SimisUlrichVasconcelos03}
	Aron Simis, Bernd Ulrich, and Wolmer~V. Vasconcelos.
	\newblock Rees algebras of modules.
	\newblock {\em Proc. London Math. Soc. (3)}, 87(3):610--646, 2003.
	
	\bibitem{Sosa14}
	Gabriel {Sosa}.
	\newblock {On the Koszulness of multi-Rees algebras of certain strongly stable
		ideals}.
	\newblock {\em arXiv e-prints}, page arXiv:1406.2188, June 2014.
	
	\bibitem{S96}
	Bernd Sturmfels.
	\newblock {\em Gr\"{o}bner bases and convex polytopes}, volume~8 of {\em
		University Lecture Series}.
	\newblock American Mathematical Society, Providence, RI, 1996.
	
	\bibitem{W80}
	Neil~L. White.
	\newblock A unique exchange property for bases.
	\newblock {\em Linear Algebra Appl.}, 31:81--91, 1980.
	
\end{thebibliography}

\end{document}